\documentclass{amsart}

\usepackage{amssymb}
\usepackage{graphicx}
\usepackage[cmtip,all]{xy}
\usepackage{hyperref}

\newtheorem{theorem}{Theorem}[section]

\newtheorem{corollary}[theorem]{Corollary}
\newtheorem{prop}[theorem]{Proposition}
\theoremstyle{definition}
\newtheorem{definition}[theorem]{Definition}
\newtheorem{example}[theorem]{Example}
\newtheorem{eser}[theorem]{Exercise}

\theoremstyle{remark}
\newtheorem{remark}[theorem]{Remark}

\numberwithin{equation}{section}

\newcommand{\Hom}{\operatorname{Hom}}
\newcommand{\Ker}{\operatorname{Ker}}
\newcommand{\Ext}{\operatorname{Ext}}
\newcommand{\Leq}{\leq_{\textrm{deg}}}
\newcommand{\ZZ}{\mathbb{Z}}
\newcommand{\PP}{\textbf{P}}
\newcommand{\CC}{\mathbb{C}}
\newcommand{\Gr}{\textrm{Gr}}

\begin{document}

\title{Three lectures on Quiver Grassmannians}

\author{Giovanni Cerulli Irelli}
\address{Via Antonio Scarpa 10, 00163, Roma (ITALY)}
\curraddr{}
\email{giovanni.cerulliirelli@uniroma1.it}
\thanks{}

\subjclass[2000]{Primary }

\date{12th August 2018}

\begin{abstract}
This paper contains the material discussed in the series of three lectures that I gave during the workshop of the ICRA 2018 in Prague. I will introduce the reader to some of the techniques used in the study of the geometry of quiver Grassmannians. The notes are quite elementary and thought for phd students or young researchers. I assume that the reader is familiar with the representation theory of quivers. 
\end{abstract}

\maketitle
\section*{Introduction}

Given a finite quiver $Q$ and a finite dimensional $Q$--representation $M$, the quiver Grassmannian $\Gr_\mathbf{e}(M)$ is the projective variety of $Q$--subrepresentations $N\subseteq M$ of dimension vector $\mathbf{dim}\,N=\mathbf{e}$. Quiver Grassmannians were considered in the seminal paper of Schofield \cite{S} for the study of general representations of $Q$. It is shown there that a general representation of dimension vector $\mathbf{d}$ admits a subrepresentation of dimension vector $\mathbf{e}$ if and only if the minimal value of the dimension of the extension space between a representation of dimension vector $\mathbf{e}$ and one of dimension vector $\mathbf{d-e}$ is zero. This is shown by considering a universal family 
$$
\pi_\mathbf{e,d}:\mathcal{Y}_\mathbf{e,d}\rightarrow R_\mathbf{d}(Q)
$$
over the representation space $R_\mathbf{d}(Q)$ of $Q$--representations of dimension vector $\mathbf{d}$ whose fiber over a point $M\in R_\mathbf{d}(Q)$ is $\Gr_\mathbf{e}(M)$. This is a proper family whose total space is smooth and irreducible. It is nowadays called the universal quiver Grassmannian. 

Quiver Grassmannians then appeared in the Fomin and Zelevinsky theory of cluster algebras (\cite{FZI,FZII,FZIV}), by work of Caldero-Chapoton \cite{CC}, Caldero-Keller \cite{CK1,CK2} and Derksen-Weyman-Zelevinsky \cite{DWZ2}. In those papers it is shown that for any (non-initial) generator $u$ of the cluster algebra $\mathcal{A}_Q$ associated with $Q$ there exists a $Q$--representation $M$ such that $u$ has the following form
\begin{equation}\label{Eq:CC-Formula}
u=\mathbf{x}^{\mathbf{g}_M}\left(\sum_\mathbf{e}\chi(\Gr_\mathbf{e}(M))\mathbf{x}^{B\mathbf{e}}\right)
\end{equation}
where $\mathbf{g}_M$ is the index of $M$ and $\chi$ denotes the Euler characteristic (see section~\ref{Sec:CCMap}). This is a remarkable fact, because the generators of the cluster algebra $\mathcal{A}_Q$ are defined recursively, starting from the initial seed $(B_Q,\mathbf{x})$. Thus, formula \eqref{Eq:CC-Formula} is a solution of this complicated recurrence relation and it is given in terms of quiver Grassmannians. It is then a natural question to see if a better understanding of the geometry of the projective variety $\Gr_\mathbf{e}(M)$ can provide useful information about the cluster algebra $\mathcal{A}_Q$. This turned out to be true in the affine type $A$: in \cite{CEsp} and \cite{CeDEsp} it is shown that by considering only the smooth part of $\Gr_\mathbf{e}(M)$ in the formula \eqref{Eq:CC-Formula}  one gets the elements of the atomic basis of $\mathcal{A}_Q$. The atomic basis is a $\ZZ$--basis such that its positive span coincides with the set of elements of $\mathcal{A}_Q$ which have positive coefficients with respect to any cluster. Nowadays it is known that an atomic basis exists only in very particular cases. 

When the quiver is acyclic, in formula \eqref{Eq:CC-Formula} the $Q$--representation $M$ is rigid, i.e. $\Ext^1(M,M)=0$. 
The positivity conjecture of Fomin and Zelevinky hence implies that the Euler characteristic of the quiver Grassmannians associated with a rigid $Q$-representation is non-negative. This was proved by Nakajima \cite{Naka}. Caldero-Keller and others conjectured that much more is true, namely that those quiver Grassmannians admit a cellular decomposition. This conjecture is still open. In \cite{CEFR} it is proved a little less: namely that those quiver Grassmannians have property (S), i.e. no odd homology, no torsion in even homology, and the cycle map is an isomorphism. This refines the proof of Nakajima. For Dynkin and affine quivers much more is true: for Dynkin quivers every quiver Grassmannian admits a cellular decomposition and for affine quivers, every quiver Grassmannian associated with a representation $M$ whose regular part is rigid, admits a cellular decomposition. See section~\ref{Sec:CellDec}. 

Apart from this motivation, the geometry of quiver Grassmannians is an interesting object of study, due to the fact that many geometric properties can be studied via the representation theory of quivers. But one has to be careful here: Reineke showed that every projective variety can be realized as a quiver Grassmannian in an elementary way  and Ringel straightened considerably this result by showing that every projective variety arises as a quiver Grassmannian of every wild quiver. See section~\ref{Sec:EveryProj} for this, and for some examples.

It is then natural to restrict attention to particular quivers and dimension vectors.  The most fruitful restriction is when $Q$ is an equioriented quiver of type $A_n$, $\mathbf{d}=(n+1,\cdots, n+1)$ and $\mathbf{e}=(1,2,\cdots, n)$. In this case the generic fiber of the universal quiver Grassmannian is the complete flag variety for $SL_{n+1}$ and the other fibers can be hence considered as ``linear'' degenerations of the complete flag variety. Among all fibers one is of particular interest: the Feigin degenerate flag variety. In \cite{CFR, CFR2, CFR3, CFR4}, we have studied degenerate flag varieties (and more general quiver Grassmannians of Dynkin type) from this point of view and get interesting new results and new proofs of known results. In \cite{CFFFR} we have explored the universal quiver Grassmannian for the special case mentioned above, and find a very interesting variety which is a flat degeneration of the complete flag variety and having the n-th Catalan number of irreducible components. See Section~\ref{Sec:TypeA} for details concerning quiver Grassmannians of type $A$, linear degenerations of flag varieties and quiver Grassmannians of Dynkin type.

In the last section~\ref{Sec:Exercises} a collection of exercises is provided. The exercises are divided according to the different sections of the paper and they are thought to provide a better understanding of the techniques mentioned in the main body of the paper.  I encourage the reader to  solve the exercises corresponding to a given section during the study of the section. 

I would like to thank the organizers of the ICRA 2018 for inviting me to give a series of lectures on this topic. I also want to thank the Ph.D. students who asked several questions during the lectures; I hope that this paper can serve to them as a handy guide into this subject. I am indebted to all my coauthors, in particular  Markus Reineke and Evgeny Feigin, on whose work most of this paper is based on.  Finally, I sincerely thank the anonymous referee and Alex Puntz for a careful reading of a previous version of this paper and for many helpful suggestions. 

The final version of this manuscript will appear in the proceedings of ICRA 2018, that will be published by AMS in Contemporary  Mathematics.

\section{Notations}\label{Sec:1}
Let $Q$ be a finite acyclic and connected quiver.  We denote by $Q_0$ the  finite set of vertices (whose cardinality is always denoted with the letter $n$), by $Q_1$ the finite set of edges, and the two functions  $s,t: Q_1\rightarrow Q_0$ provide an orientation of the edges. For an oriented edge  $\alpha$ we write $\alpha:s(\alpha)\rightarrow t(\alpha)$. 
The base field is the field of complex numbers, denoted either by $K$ or with the usual symbol $\CC$. We denote by $\textrm{Rep}(Q)$ the category of finite-dimensional complex representations of $Q$. Recall that the objects of $\textrm{Rep}(Q)$ are  tuples $M=((M_i)_{i\in Q_0}, (M_\alpha)_{\alpha\in Q_1})$ where $M_i$ is a (finite-dimensional) vector space and $M_\alpha:M_{s(\alpha)}\rightarrow M_{t(\alpha)}$  is a linear map.  
A $Q$--morphism $\psi:M\rightarrow N$ between two $Q$--representations is a collections $(\psi_i:M_i\rightarrow N_i)_{i\in Q_0}$ of linear maps such that the following square 
$$
\xymatrix{
M_{s(\alpha)}\ar^{M_{\alpha}}[r]\ar_{\psi_{s(\alpha)}}[d]&M_{t(\alpha)}\ar^{\psi_{t(\alpha)}}[d]\\
N_{s(\alpha)}\ar^{N_{\alpha}}[r]&N_{t(\alpha)}
}
$$
commutes for every arrow $\alpha$ of $Q$. We denote by $\Hom_Q(M,N)$ the vector space of $Q$--morphisms between the two $Q$--representations $M$ and $N$. We denote its dimension by 
$$
[M,N]:=\textrm{dim}\,\Hom_Q(M,N). 
$$
To a quiver $Q$ is associated its (complex) path-algebra $A=KQ$, which is the algebra formed by concatenation of arrows. The category $\textrm{Rep}_K(Q)$ is equivalent to the category $A$--mod of $KQ$--modules. Notice that $KQ$ is finite--dimensional since the quiver $Q$ is acyclic, i.e. it does not have  oriented cycles (even if its underlying graph can have a cycle).  The category $\textrm{Rep}_K(Q)$ is abelian and  Krull-Schmidt, moreover it is hereditary, i.e. $\Ext^{\geq 2}_Q(-,-)=0$. We use the standard notation 
$$
[M,N]^1:=\textrm{dim}\,\Ext^1_Q(M,N). 
$$

The set $\{e_i\}_{i\in Q_0}$ of paths of length zero form a complete set of pairwise orthogonal idempotents of $A$. Since $Q$ is acyclic, and hence the path algebra $A=KQ$ is finite-dimensional, there are only finitely many simple $A$--modules parametrized by the vertices of $Q$. We denote by $S_k$ the simple corresponding to vertex $k$, by $P_k$ its projective cover and by $I_k$ it injective hull. Recall that as  $Q$--representation, $P_k$ is described as follows: the vector space at vertex $i$ has a basis given by paths from vertex $k$ to vertex $i$, and the arrows act by ``concatenation''. Notice that if $Q$ is an orientation of a tree (for example if $Q$ is Dynkin), then every projective $P_k$ is thin, which means that the vector space $(P_k)_i$ at every vertex $i$ is at most one--dimensional. 
Dually, the injective indecomposable (left) $A$--modules are the indecomposable direct summands of $DA$ (viewed as left $A$--module), where $D$ is the standard $K$--duality.  As $Q$--representation, $I_k$ has at vertex $j$ a vector space with basis consisting of all the paths of $Q$ starting in $j$ and ending in $k$, and the arrows act by ``concatenation''.  

For a $Q$--representation $M$, the collection  $(\dim M_i)_{i\in Q_0}\in\ZZ^{Q_0}_{\geq0}$ of non--negative integers is called the \emph{dimension vector} of $M$, and it is denoted in bold by $\mathbf{dim}\, M$. Once the dimension vector is fixed, a $Q$--representation is determined by linear maps: this leads us to the variety of $Q$--representations. Let $\mathbf{d}=(d_i)_{i\in Q_0}\in\ZZ^{Q_0}_{\geq0}$ be a dimension vector. The vector space 
$$
R_\mathbf{d}:=\bigoplus_{\alpha\in Q_1} \Hom_K(K^{d_{s(\alpha)}},K^{d_{t(\alpha)}})
$$
is called the variety of $Q$--representations of dimension vector $\mathbf{d}$. The group
$$
G_\mathbf{d}:=\prod_{i\in Q_0}\textrm{GL}_{d_i}(K)
$$
acts on $R_\mathbf{d}$ by base change:
$
(g_i)_i\cdot (V_\alpha)_\alpha:=(g_{t(\alpha)}V_\alpha g_{s(\alpha)}^{-1})_\alpha
$
and $G_\mathbf{d}$--orbits are in bijection with isoclasses of $Q$--representations. The stabilizer of a point  $M\in R_\mathbf{d}$ is  
$$
\textrm{Stab}_{G_\mathbf{d}}(M)=\textrm{Aut}_Q(M)
$$
where $\textrm{Aut}_Q(M)$ denotes the open subvariety of $\Hom_Q(M,M)$ consisting of invertible $Q$-morphisms. In particular, $\textrm{dim }\textrm{Aut}_Q(M)=\textrm{dim }\Hom_Q(M,M)$. 
Given another dimension vector $\mathbf{e}\in\ZZ^{Q_0}_{\geq0}$ we consider the vector space of (``degree zero'') $K$--morphisms
$$
\Hom(\mathbf{e},\mathbf{d})=\bigoplus_{i\in Q_0} \Hom_K(K^{e_i},K^{d_i})
$$ 
and the vector space of (``degree one'') $K$--morphisms
$$
\Hom(\mathbf{e},\mathbf{d}[1])=\bigoplus_{\alpha \in Q_1} \Hom_K(K^{e_{s(\alpha)}},K^{d_{t(\alpha)}}).
$$ 
In particular, if $\mathbf{d}$ is a dimension vector we get
\begin{equation}\label{Eq:DimRd}
\dim \Hom(\mathbf{d},\mathbf{d}[1])=\dim R_\mathbf{d}.
\end{equation}

Given  $N\in R_\mathbf{e}$ and $M\in R_\mathbf{d}$ we consider the map
$$
\Phi^M_N:\Hom(\mathbf{e},\mathbf{d})\rightarrow \Hom(\mathbf{e},\mathbf{d}[1]):\;(f_i)_{i\in Q_0}\mapsto (M_\alpha\circ f_{s(\alpha)}-f_{t(\alpha)}\circ N_\alpha)_{\alpha\in Q_1}
$$
This is a linear map between finite dimensional vector spaces and one can show quite easily (see e.g. \cite{R},  \cite{ASS}): 
$$
\begin{array}{cc}
\Ker \Phi_N^M=\Hom_Q(N,M),&\textrm{CoKer}\, \Phi_N^M\simeq \Ext^1_Q(N,M).
\end{array}
$$
From these formulas we immediately get:
\begin{equation}\label{Eq:Euler}
[N,M]- [N,M]^1=\dim \Hom(\mathbf{e},\mathbf{d})-\dim\Hom(\mathbf{e},\mathbf{d}[1]).
\end{equation}
We have that $\dim \Hom(\mathbf{e},\mathbf{d})=\sum_{i\in Q_0}\!\!e_id_i$ and $\dim\Hom(\mathbf{e},\mathbf{d}[1])=\sum_{\alpha\in Q_1}\!\!e_{s(\alpha)}d_{t(\alpha)}$. Given two arbitrary integer vectors $\mathbf{e},\mathbf{d}\in \ZZ^{Q_0}$ the Euler form of $Q$ is the integral bilinear form
$\langle-,-\rangle_Q:\ZZ^{Q_0}\times\ZZ^{Q_0}\rightarrow \ZZ$
 given by
$$
\langle\mathbf{e},\mathbf{d}\rangle:=\sum_{i\in Q_0}e_id_i-\sum_{\alpha\in Q_1}e_{s(\alpha)}d_{t(\alpha)}.
$$
From \eqref{Eq:Euler} above, we immediately get
\begin{equation}\label{Eq:EulForm}
\dim \Hom_Q(N,M)- \dim \Ext^1_Q(N,M)=\langle\mathbf{dim}\,N,\mathbf{dim}\,M\rangle.
\end{equation}
Formula~\eqref{Eq:EulForm} is called the homological interpretation of the Euler form.

%
In view of  \eqref{Eq:DimRd} and \eqref{Eq:EulForm}, we have
$$
\textrm{codim}_{R_\mathbf{d}}\,(G_\mathbf{d}\cdot M)=\textrm{dim}\,R_\mathbf{d}-\textrm{dim Stab}_{G_\mathbf{d}}(M)=\textrm{dim}\,\Ext^1_Q(M,M).
$$
We conclude that the orbit of $M$ is dense in $R_\mathbf{d}$ if and only if $\Ext^1_Q(M,M)=0$. A representation $M$ such that $\Ext^1_Q(M,M)=0$ is called \emph{rigid}.

A famous theorem of P.~Gabriel \cite{Gabriel} (see also \cite{BGP} for a different proof and \cite[Section~VII.5]{ASS} for a survey) states that a quiver $Q$ admits only a finite number of isoclasses of indecomposable representations if and only if $Q$ is a Dynkin quiver i.e. it is an orientation of a simply-laced Dynkin diagram of type $A,D,E$. The quiver $Q$ is called tame or affine if it is an acyclic orientation of a simply-laced extended Dynkin diagram of type $ADE$.  A quiver which is neither Dynkin nor affine is called wild. The classification of the indecomposable $Q$--representations is possible if and only if $Q$ is either Dynkin or tame and this explains the terminology. Table~\ref{Fig:ExtendedDynkinDiagrams} shows the Dynkin and the extended Dynkin diagrams.  
\begin{table}[htbp]
\begin{center}
$$
\begin{array}{|c|c|c|}
\hline
\textrm{Type}&\textrm{Dynkin}&\textrm{Extended Dynkin}\\\hline
\hline
\xymatrix@R=10pt{\\A}
&
\xymatrix@C=8pt@R=8pt{&&&&\\
\bullet\ar@{-}[r]\ar@{-}[rr]&\bullet\ar@{-}[r]&\cdots\ar@{-}[r]                   &\bullet\ar@{-}[r]&\bullet
                          }
                         &
\xymatrix@C=8pt@R=8pt{
                         &                                         &\bullet\ar@{-}[drr]                                           &                                           &\\
\bullet\ar@{-}[r]\ar@{-}[urr]&\bullet\ar@{-}[r]&\cdots\ar@{-}[r]                   &\bullet\ar@{-}[r]&\bullet
                          }
\\\hline
\xymatrix@R=10pt{\\D}&
                          \xymatrix@C=8pt@R=5pt{
&                         &                        &                                           &                                           &\bullet\\
                          &\bullet\ar@{-}[r]&\bullet\ar@{-}[r]&\cdots\ar@{-}[r]                   &\bullet\ar@{-}[ur]\ar@{-}[dr]&\\
\bullet\ar@{-}[ur]&                        &                         &                                           &                                           &\bullet
}
&
                          \xymatrix@C=8pt@R=5pt{
\bullet\ar@{-}[dr]&                         &                        &                                           &                                           &\bullet\\
                          &\bullet\ar@{-}[r]&\bullet\ar@{-}[r]&\cdots\ar@{-}[r]                   &\bullet\ar@{-}[ur]\ar@{-}[dr]&\\
\bullet\ar@{-}[ur]&                        &                         &                                           &                                           &\bullet
}
\\\hline
\xymatrix@R=10pt{\\E_6}&
\xymatrix@C=10pt@R=10pt{
                        &                    &\bullet           &     &\\
\bullet\ar@{-}[r]&\bullet\ar@{-}[r]&\bullet\ar@{-}[r]\ar@{-}[u]&\bullet\ar@{-}[r]&\bullet}
&
\xymatrix@C=10pt@R=10pt{
                        &   \bullet    \ar@{-}[r]                  &\bullet           &     &\\
\bullet\ar@{-}[r]&\bullet\ar@{-}[r]&\bullet\ar@{-}[r]\ar@{-}[u]&\bullet\ar@{-}[r]&\bullet}
\\\hline
\xymatrix@R=10pt{\\E_7}&
\xymatrix@C=10pt@R=10pt{
&& &\bullet           &  &   &\\
&\bullet\ar@{-}[r]&\bullet\ar@{-}[r]&\bullet\ar@{-}[r]\ar@{-}[u]&\bullet\ar@{-}[r]&\bullet\ar@{-}[r]&\bullet}
&
\xymatrix@C=10pt@R=10pt{
&& &\bullet           &  &   &\\
\bullet\ar@{-}[r]&\bullet\ar@{-}[r]&\bullet\ar@{-}[r]&\bullet\ar@{-}[r]\ar@{-}[u]&\bullet\ar@{-}[r]&\bullet\ar@{-}[r]&\bullet}
\\\hline\xymatrix@R=10pt{\\E_8}&
\xymatrix@C=8pt@R=10pt{
&&& &           & \bullet &   &\\
&\bullet\ar@{-}[r]&\bullet\ar@{-}[r]&\bullet\ar@{-}[r]&\bullet\ar@{-}[r]&\bullet\ar@{-}[u]\ar@{-}[r]&\bullet\ar@{-}[r]&\bullet}
&
\xymatrix@C=10pt@R=10pt{
&&& &           & \bullet &   &\\
\bullet\ar@{-}[r]&\bullet\ar@{-}[r]&\bullet\ar@{-}[r]&\bullet\ar@{-}[r]&\bullet\ar@{-}[r]&\bullet\ar@{-}[u]\ar@{-}[r]&\bullet\ar@{-}[r]&\bullet}
\\\hline
\end{array}
$$
\caption{Dynkin and extendend Dynkin diagrams}
\label{Fig:ExtendedDynkinDiagrams}
\end{center}
\end{table}

If $Q$ is Dynkin, then $R_\mathbf{d}$ consists of finitely many $G_\mathbf{d}$--orbits, and hence, since such orbits are connected and locally closed, there is a unique orbit which is dense. The corresponding representation is hence a generic representation of dimension vector 
$\mathbf{d}$ and we denote it by $\tilde{M}_\mathbf{d}$.  In particular, for Dynkin quivers a representation is generic if and only if it is rigid. 

For an arbitrary acyclic quiver $Q$ most dimension vectors do not admit a dense orbit. By Kac's theorem (\cite[Theorem~1]{Kac1}) there exists an indecompsable representation of  dimension vector $\mathbf{d}$ if and only if $\mathbf{d}$ is a positive root for the Kac-Moody algebra associated with the underlying graph of $Q$; in this case $R_\mathbf{d}$ admits a dense orbit if and only if $\mathbf{d}$ is a positive real root. 


\subsection{Almost split sequences}
We conclude this section by recalling the fundamental notions of \emph{almost split sequence}, \emph{irreducible morphism} and of \emph{Auslander--Reiten quiver} of a quiver $Q$ (see e.g. \cite{CB2}, \cite{ARS}, \cite{ASS}).   A short exact sequence 
$$
\xymatrix{
\delta:&0\ar[r]&N\ar^f[r]&E\ar^g[r]&M\ar[r]&0
}
$$ 
is called \emph{almost split} if it is non--split, both $N$ and $M$ are indecomposable and for any morphism $h:X\rightarrow M$ which is not a split epi (i.e. it does not admit  a right inverse), there exists $t:X\rightarrow E$ such that $h=g\circ t$. In particular, if $\delta$ is an almost split sequence, and $M$ is not a direct summand of $X$, then $[X,E]=[X,N\oplus M]$. 
Dually, it can be shown that $\delta$ is almost split if and only if it is non--split, both $N$ and $M$ are indecomposable and for any morphism $h:N\rightarrow X$ which is not a split mono (i.e. it does not admit a left inverse), there exists $t:E\rightarrow X$ such that $h=t\circ f$. 
A fundamental result of Auslander and Reiten \cite[Theorem~V.1.15]{ARS} states that for every indecomposable $M$ which is not projective, there exists an almost split sequence $\delta$ as above (ending in $M$), which is unique up to scalar multiples \cite[Theorem~V.1.16]{ARS}. Dually, for every indecomposable $N$ which is not injective, there exists an almost split sequence $\delta$ as above (starting from $N$). 

One can show that almost split sequences are \emph{rigid}, in the sense that they are uniquely determined (up to scalar multiples as elements of $\Ext^1(N,M)$) by the three modules $N$, $E$ and $M$ \cite[Proposition~V.2.3]{ARS}. 

Almost split sequences are closely related to the so-called Auslander-Reiten translate $\tau$ and its quasi-inverse $\tau^-$. In general the definition of $\tau$ and $\tau^-$ is quite involved since they are not  functors, but in our situation, which is the case of an hereditary basic and finite dimensional algebra, it reduces to two  simple functors:  
$$
\begin{array}{cc}
\tau=D\Ext^1(-,A)&\tau^-=\Ext^1(D(-),A).
\end{array}
$$
They are uniquely determined by the Auslander-Reiten formulas: 
\begin{equation}\label{ARFormulas}
\Hom(M,\tau N)\simeq D\Ext^1(N,M)\simeq\Hom(\tau^-M, N).
\end{equation}
If there is an amost split sequence $\delta$ as above then
$$
\begin{array}{cc}
N\simeq\tau $M$,&M\simeq\tau^-N.
\end{array}
$$
\subsection{Auslander-Reiten quiver}
A morphism $f:M\rightarrow N$ between two indecomposable $Q$--representations $M$ and $N$ is called \emph{irreducible} if $f$ is neither split mono, nor split epi (i.e. it does not admit neither a left nor a right inverse) and whenever there is a factorization $f=f_2\circ f_1$, then either $f_1$ is  split mono  or $f_2$ is  split epi (see \cite{CB2}). The irreducible morphisms from $M$ to $N$ are parametrized by the quotient space $\textrm{Irr}(M,N)=\textrm{rad}_Q(M,N)/\textrm{rad}^2(M,N)$ (see e.g. \cite[Section~1]{CB2}). Here $\textrm{rad}(M,N)=\{f:M\rightarrow N\textrm{ not an isomorphism}\}$ and $\textrm{rad}^2(M,N)=\{f:M\rightarrow N \textrm{which factor as hg with g not split mono and h not split epi}\}$. The Auslander--Reiten quiver of $Q$ is a quiver denoted by $\Gamma_Q$ whose vertices are isoclasses of indecomposable $Q$--representations, and there are $k$ arrows $[M]\rightarrow [N]$ if the dimension of the quotient space $\textrm{Irr}(M,N)$ has dimension $k$ (see e.g. \cite[Section~1]{CB2} or \cite[Section~IV.4]{ASS}). 

The Auslander-Reiten quiver $\Gamma_Q$ consists of several connected components which can be grouped together to form a decomposition 
$$
\Gamma_Q=\mathcal{P}_Q\amalg \mathcal{R}_Q\amalg \mathcal{I}_Q
$$
where $\mathcal{P}_Q$ is the component containg all the indecomposable projectives  and dually,  $\mathcal{I}_Q$ is the component containg all the indecomposable injectives. The remaining piece $\mathcal{R}_Q$ consists of all connected components which do not contain neither an injective nor a projective module. It is easy to see that both $\mathcal{P}_Q$ and $\mathcal{I}_Q$ are connected components.  An indecomposable module $M$ lies in $\mathcal{P}_Q$ (resp. $\mathcal{I}_Q$) if and only if there exists a vertex $k\in Q_0$ and an index $j\geq0$ such that $M\simeq \tau^{-j}P_k$ (resp. $M\simeq \tau^{j}I_k$). Such a module is called preprojective (resp. preinjective). With abuse of notation, we denote by $\mathcal{P}_Q$ (resp. $\mathcal{I}_Q$) the full additive subcategory of $\mathrm{Rep}(Q)$ whose indecomposable objects are preprojectives (resp. preinjectives) and its objects are still called preprojectives (resp. preinjectives). The connected components of $\mathcal{R}_Q$ are called regular and their modules are also called regular.

The components $\mathcal{P}_Q$ and $\mathcal{I}_Q$ are called the preprojective and preinjective component of $\Gamma_Q$, respectively. These components of the quiver $\Gamma_Q$ can be described combinatorially via the knitting algorithm.
(I recommend the introductory book \cite{Ralf} for more details about the construction of AR quivers of Dynkin quivers.) 

The main property of the graphs $\mathcal{P}_Q$ and $\mathcal{I}_Q$ is that they encode the information that one needs to understand homomorphisms and extensions between the indecomposable representations corresponding to their vertices. Namely, the dimension of $\Hom(M,N)$ is given by counting paths from $[M]$ to $[N]$ modulo mesh relations. To get the extension spaces one uses the Auslander-Reiten formulas \eqref{ARFormulas}. Moreover $\mathcal{P}_Q$ is a directed category,  in the sense that given two indecomposables $M,N\in\mathcal{P}_Q$ either $[M,N]^1=0$ or $[N,M]^1=0$. The same holds for $\mathcal{I}_Q$.  The regular components encode much less information due to the fact that the infinite radical of the module category contains many maps (if $Q$ is wild). They are described by Ringel \cite{Ringel:Wild}. They are far from being directed (see \cite{Kerner}).

Every module $M$ admits a unique split filtration
$M'\subseteq M''\subseteq M$
where $M'\in\mathcal{I}_Q$, $M''/M'\in\mathcal{R}_Q$ and $M/M''\in\mathcal{P}_Q$; these are called the preinjective, regular and preprojective parts of $M$, respectively.

\section{Quiver Grassmannians}\label{Sec:QG}
Let $Q$ be a finite quiver  with $n$ vertices and let $A=KQ$ be the associated (complex) path algebra. Given a dimension vector $\mathbf{d}$, an  $A$--module $M\in R_\mathbf{d}$ and another dimension vector $\mathbf{e}$ such that $\mathbf{d-e}\in\ZZ^{Q_0}_{\geq0}$, in this section we define the projective variety $Gr_\mathbf{e}(M)$ whose points parametrize submodules of $M$ of dimension vector $\mathbf{e}$.  We need to ask ourselves ``what is a submodule?''. This question has two answers: first of all, a submodule is a collection $(N_i)_{i\in Q_0}$ of vector subspaces $N_i\subseteq M_i$ such that $M_\alpha(N_i)\subseteq N_j$ for every arrow $\alpha:i\rightarrow j$ of Q. On the other hand, a submodule $N\subset M$ is an $A$--module $N$ endowed with an injective $A$--morphism $\iota: N\rightarrow M$.  The two answers provide two different realizations of $Gr_\mathbf{e}(M)$.

\subsection{First realization: universal quiver Grassmannians}
Schofield noticed that quiver Grassmannians come in families:  Let $\mathbf{d}$ and $\mathbf{e}$ be two dimension vector for $Q$ such that $e_i\leq d_i$ for all $i\in Q_0$. Let us consider the product of usual Grassmannians of  vector spaces over the field $K$ of complex numbers: $Gr_\mathbf{e}(\mathbf{d}):=\prod_{i\in Q_0}Gr_{e_i}(K^{d_i})$. Given $M\in R_\mathbf{d}(Q)$ and a point $N\in Gr_\mathbf{e}(\mathbf{d})$, the condition that $N$ defines a sub-representation of $M$ is $M_\alpha(N_{s(\alpha)})\subseteq N_{t(\alpha)}$. We hence consider the incidence variety inside $Gr_\mathbf{e}(\mathbf{d})\times R_\mathbf{d}$ given by:
\begin{equation}\label{Eq:DefUnivQG}
Gr_{\mathbf{e}}^{Q}(\mathbf{d}):=\{(N,M)\in Gr_\mathbf{e}(\mathbf{d})\times R_\mathbf{d}|\, M_\alpha(N_{s(\alpha)})\subseteq N_{t(\alpha)},\,\forall \alpha\in Q_1\}.
\end{equation}
The two projections $p_1:Gr_\mathbf{e}(\mathbf{d})\times R_\mathbf{d}\rightarrow Gr_\mathbf{e}(\mathbf{d})$ and $p_2:Gr_\mathbf{e}(\mathbf{d})\times R_\mathbf{d}\rightarrow R_\mathbf{d}$ induce two maps 
$$
\xymatrix{
&Gr_\mathbf{e}^Q(\mathbf{d})\ar_{p_\mathbf{e}}@{->}[dl]\ar^{p_\mathbf{e,d}}@{->}[dr]&\\
Gr_\mathbf{e}(\mathbf{d})&&R_\mathbf{d}
}
$$
The group $G_\mathbf{d}$ acts diagonally on $Gr_{\mathbf{e}}^{Q}(\mathbf{d})$ (see exercise 9.5) and the two maps $p_\mathbf{e}$ and $p_\mathbf{e,d}$ are $G_\mathbf{d}$--equivariant. The map $p_2$ is proper; moreover $Gr_\mathbf{e}^Q(\mathbf{d})$ is closed in $Gr_\mathbf{e}(\mathbf{d})\times R_\mathbf{d}$ and the closed embedding $Gr_\mathbf{e}^Q(\mathbf{d})\rightarrow Gr_\mathbf{e}(\mathbf{d})\times R_\mathbf{d}$ is proper. It follows that the map $p_\mathbf{e,d}$ is proper, being the composition of two proper maps.  Its image is the \emph{closed} subset of $R_\mathbf{d}$ consisting of those points $M\in R_\mathbf{d}$ which admit a sub-representation of dimension vector $\mathbf{e}$. The \emph{quiver Grassmannian} $Gr_\mathbf{e}(M)$ associated with a point $M\in R_\mathbf{d}$ is defined as the (scheme-theoretic) fiber of $p_\mathbf{e,d}$ over $M$. 

Thus quiver Grassmannians come in families: they are fibers of the proper map 
$$
p_\mathbf{e,d}:Gr_\mathbf{e}^Q(\mathbf{d})\rightarrow R_\mathbf{d} 
$$
which is called the universal quiver Grassmannian. 
As shown in \cite[section~2.2]{CFR}, the map $p_\mathbf{e}$ realizes $Gr_\mathbf{e}^Q(\mathbf{d})$ as the total space of an homogeneous vector bundle over $Gr_\mathbf{e}(\mathbf{d})$ of rank \[\sum_{\alpha\in Q_1}d_{s(\alpha)}d_{t(\alpha)}+e_{s(\alpha)}e_{t(\alpha)}-e_{s(\alpha)}d_{t(\alpha)}.\] In particular, $Gr_\mathbf{e}^Q(\mathbf{d})$ is smooth and irreducible of dimension
$$
\textrm{dim }Gr_\mathbf{e}^Q(\mathbf{d})=\langle\mathbf{e},\mathbf{d}-\mathbf{e}\rangle+\textrm{dim }R_\mathbf{d}.
$$
By upper--semicontinuity of the fiber dimension, we see that for any point $M$ in the image of $p_\mathbf{e,d}$ we have
\begin{equation}\label{Eq:IneqDimQuivGrass}
\textrm{dim }Z\geq \textrm{dim }Gr_\mathbf{e}^Q(\mathbf{d})-\textrm{dim }Im(p_\mathbf{e,d})\geq \langle\mathbf{e},\mathbf{d}-\mathbf{e}\rangle.
\end{equation}
for every irreducible component $Z$ of $Gr_\mathbf{e}(M)$.

Let $D:Rep(Q)\rightarrow Rep(Q^{op})$ be the standard duality which associates to a $Q$--representation $M$ its linear dual $DM$. There is an isomorphism of projective varieties
\begin{equation}\label{Eq:DualityGrass}
\zeta:Gr_\mathbf{e}(M)\rightarrow Gr_{\mathbf{d}-\mathbf{e}}(DM):\, L\mapsto \textrm{Ann}_M(L):=\{\varphi\in DM|\, \varphi(L)=0\}.
\end{equation}
%
\subsection{Second realization: quiver Grassmannians as geometric quotients and stratification}\label{Sec:Stratification}
Following Caldero and Reineke \cite{CR}, one  can realize quiver Grassmannians as geometric quotients.
Recall the two vector spaces $\Hom(\mathbf{e},\mathbf{d})$ and $\Hom(\mathbf{e},\mathbf{d}[1])$ of section~\ref{Sec:1} and the linear map $\Phi_L^M:\Hom(\mathbf{e},\mathbf{d})\rightarrow \Hom(\mathbf{e},\mathbf{d}[1])$ associated with $L\in R_\mathbf{e}(Q)$ and $M\in R_\mathbf{d}(Q)$. Let us assume that $e_i\leq d_i$ for all $i\in Q_0$. Given $M\in R_\mathbf{d}(Q)$ the algebraic map
$$
\Phi^M: R_\mathbf{e}\times \Hom(\mathbf{e},\mathbf{d})\rightarrow \Hom(\mathbf{e},\mathbf{d}[1]): (L,f)\mapsto \Phi^M_L(f)
$$
is used to define the following closed subvariety of $R_\mathbf{e}\times \Hom(\mathbf{e},\mathbf{d})$:
$$
\Hom(\mathbf{e},M):=\{(L,f)\in  R_\mathbf{e}\times \Hom(\mathbf{e},\mathbf{d})|\, \Phi^M_L(f)=0\}.
$$
Inside $\Hom(\mathbf{e},\mathbf{d})$ there is the open (and dense) subvariety $\Hom^0(\mathbf{e},\mathbf{d})$ consisting of collections of injective linear maps; the induced open subvariety  $\Hom^0(\mathbf{e},M):=\Hom(\mathbf{e},M)\cap \left(R_\mathbf{e}\times \Hom^0(\mathbf{e},\mathbf{d})\right)$ is of particular importance for us. Indeed the map 
$$
\phi: \Hom^0(\mathbf{e},M)\rightarrow Gr_\mathbf{e}(M):\, (L,f)\mapsto f(L)
$$
is surjective and each fiber of $\phi$ is a free orbit for the algebraic group $G_\mathbf{e}=\prod_{i\in Q_0}GL(e_i)$ (see \cite[Lemma~2]{CR}). This implies that the quiver Grassmannian $Gr_\mathbf{e}(M)$ is a geometric quotient:
\begin{equation}\label{Eq:QuotQuivGrass}
Gr_\mathbf{e}(M)\simeq \Hom^0(\mathbf{e},M)/G_\mathbf{e}.
\end{equation}
With this formulation, a point $p$ of $Gr_\mathbf{e}(M)$ is represented (up to the $G_\mathbf{e}$--action) by a pair $(L,\iota)$ where $L\in R_\mathbf{e}(Q)$ and $\iota:L\rightarrow M$ is an injective homomorphism of $Q$--representations; in this case we use the notation $p=[(L,\iota)]$. As shown by Caldero and Reineke, formula \eqref{Eq:QuotQuivGrass} implies the following description of the (scheme-theoretic) tangent space $T_{p}(Gr_\mathbf{e}(M))$ at a point $p$ of the quiver Grassmannian.
\begin{theorem}\label{Thm:TangentSpace}
Given $M\in R_\mathbf{d}(Q)$ and a point $p=[(L,\iota)]\in Gr_\mathbf{e}(M)$, the  tangent space $T_{p}(Gr_\mathbf{e}(M))$  
at $p$ is isomorphic to 
$\Hom_Q(L,M/\iota(L))$.
\end{theorem}
\begin{remark}
The tangent space formula only holds at level of schemes. The usual example in this sense is given by considering a regular  (indecomposable) representation $R_2$ of the Kronecker quiver of quasi--length 2 whose dimension vector is $(2,2)$. The quiver Grassmannian $Gr_{(1,1)}(R_2)$ is a point, but the tangent space has dimension one. 
\end{remark}

\begin{prop}\label{Prop:RigidQGSmooth}
A non-empty quiver Grassmannian $\Gr_\mathbf{e}(M)$ associated with a rigid representation $M$ is smooth of dimension $\langle\mathbf{e},\mathbf{dim}\,M-\mathbf{e}\rangle$. 
\end{prop}
\begin{proof}
For every subrepresentation $N\subseteq M$ we have  surjective morphisms $\xymatrix@1@C=20pt{\Ext^1(M,M)\ar@{->>}[r]&\Ext^1(N,M)\ar@{->>}[r]&\Ext^1(N,M/N)}$. This proves that if $[M,M]^1=0$ then $[N,M/N]^1=0$ and hence, by \eqref{Eq:EulForm},  the tangent space $T_{N}(Gr_\mathbf{e}(M))$ has dimension $\langle\mathbf{e},\mathbf{dim}\,M-\mathbf{e}\rangle$. This shows that $\Gr_\mathbf{e}(M)$ is smooth. The dimension is computed by using  \eqref{Eq:IneqDimQuivGrass}.
\end{proof}
As a consequence of the tangent space formula Schofield proved the following. 
\begin{theorem}\cite[Theorem~3.3]{S}\label{Thm:Schofield}
The universal quiver Grassmannian $p_\mathbf{e,d}$ is surjective if and only if there exist $N\in R_\mathbf{e}$ and $R\in R_\mathbf{d-e}$ such that $[N,R]^1=0$.
\end{theorem}

The realization of a quiver Grassmannian as a geometric quotient \eqref{Eq:QuotQuivGrass} allows to define a stratification of $Gr_\mathbf{e}(M)$ as follows  (see \cite[Section~2.3]{CFR} for more details): let $p$ be the projection from $\Hom^0_Q(\mathbf{e},M)$ to $R_\mathbf{e}$; its fiber over a point $N\in R_\mathbf{e}$ is the space of injective linear maps $\Hom^0_Q(N,M)$. For each isoclass $[N]$ in $R_\mathbf{e}$ we can consider the subset $\mathcal{S}_{[N]}$ of $Gr_\mathbf{e}(M)$ corresponding under the previous isomorphism to $p^{-1}(G_\mathbf{e}\cdot N)/G_\mathbf{e}$. The locally closed subset $\mathcal{S}_{[N]}$ is sometimes called an iso-stratum of the quiver Grassmannian. In \cite[Lemma~2.4]{CFR}  it is shown that $\mathcal{S}_{[N]}$ is a locally closed subset of dimension
$$
\textrm{dim }\mathcal{S}_{[N]}=[N,M]-[N,N].
$$
In particular, a quiver Grassmannian  $Gr_\mathbf{e}(M)$ admits a stratification
$$
Gr_\mathbf{e}(M)=\coprod_{[N]}\mathcal{S}_{[N]}.
$$
In case $M$ is preprojective this stratification is finite. In this case, the irreducible components of $Gr_\mathbf{e}(M)$ are hence closures of some strata which we called the \emph{generic sub--representation types} of $Gr_\mathbf{e}(M)$ (see \cite{CFR3}). See \cite{Hubery} for a description of the irreducible components of Grassmannians of submodules of a module over an algebra.
In case of rigid modules the following holds
\begin{theorem}{\cite[Prop.~37]{CEFR}}
A quiver Grassmannian associated with a rigid quiver representation is irreducible.
\end{theorem}

\section{Degeneration of \textrm{Q}--representations: Bongartz's theorem and applications to quiver Grassmannians}\label{Sec:Degeneration}
Given $M,N\in R_\mathbf{d}$,  $M$ is said to \emph{degenerate} to $N$ and in this case it is customary to write $M\leq_{\textrm{deg}} N$, if the closure of the orbit of $M$ contains $N$: 
$$
M\Leq N\;\;\;\stackrel{def}{\Longleftrightarrow}\;\;\; \overline{G_\mathbf{d}\cdot M}\supseteq G_\mathbf{d}\cdot N.
$$
For arbitrary finite--dimensional algebras, it is a hard problem to control such a notion. On the other hand, for algebras of finite representation type (i.e. admitting a finite number of indecomposable modules) the following very useful characterization holds: 
\begin{equation}\label{Eq:DegZwara}
\begin{array}{ccccc}
M\Leq N&\Longleftrightarrow&[X,M]\leq [X,N]&\Longleftrightarrow&[M,X]\leq [N,X].\\
&&\forall\;X\in\textrm{Rep}(Q)&&\forall\;X\in\textrm{Rep}(Q)
\end{array}
\end{equation}
For Dynkin quivers this result was obtained by Bongartz \cite{B} (partial results were obtained by Riedtmann \cite{Ried}, Abeasis-Del Fra \cite{AdF1, AdF2, AdF3}). The surprising generalization to  any algebra of finite representation type was obtained by  Zwara \cite{Z} (the second equivalence follows from Auslander--Reiten theory \cite[section~2.2]{Z}, \cite{AR85}). For general quivers, not necessarily Dynkin, the equivalence \eqref{Eq:DegZwara} holds true in case both $M$ and $N$ are  preprojective or preinjective \cite{B}. 

In the analysis of the geometry of quiver Grassmannians the following result of Bongartz can be useful. In order to formulate it we need to recall the notion of a generic quotient from Bongartz's paper \cite[Section~2.4]{B}. Suppose that  $U\in R_\mathbf{e}$ and $M\in R_\mathbf{d}$ are given, and also that there exists a monomorphism $\iota:U\rightarrow M$; in particular $\mathbf{d-e}\in\ZZ^{Q_0}_{\geq0}$ is a dimension vector. The set of all possible quotients of $M$ by $U$ is an  irreducible constructible  subset of $R_{\mathbf{d}-\mathbf{e}}$ which is $G_{\mathbf{d}-\mathbf{e}}$--invariant. If this set is the closure of the  orbit of a point $S$ then $S$ is called the \emph{generic quotient} of $M$ by $U$. In general, generic quotients may not exist. They exist for Dynkin quivers, or if $M$ is preinjective or if $M$ is regular over an affine quiver, since there are only finitely many isoclasses of quotients. 

\begin{theorem}\label{Thm:Bongartz}(\cite[Theorem~2.4]{B})
Let $M,N\in R_\mathbf{d}$ such that $M\Leq N$. Let $U$ be a representation such that $[U,M]=[U,N]$ then the following holds: 
\begin{enumerate}
\item if $U$ embeds into $N$,  it embeds into $M$ too;
\item in this case every quotient of $N$ by $U$ is  a degeneration of the generic quotient of $M$ by U, in case it exists.
\end{enumerate}
\end{theorem}

Bongartz's theorem~\ref{Thm:Bongartz} can be used to prove that a certain quiver Grassmannian is non-empty. For example it can be used to prove Schofield's theorem~\ref{Thm:Schofield}. For Dynkin quivers, an interesting homological criterion that guarantees the non--emptiness of a quiver Grassmannian associated with an arbitrary representation  can be found in  \cite{MR}. 

\section{Examples of quiver Grassmannians}
In this section we collect examples of quiver Grassmannians. 
\subsection{Example 1: The Grassmannian}
Let $Q=\cdot$ be the quiver with one vertex and no arrows. A $Q$--representation of dimension vector $\mathbf{d}=(d)\in\ZZ_{\geq0}$ is  a vector space $M=K^d$. Let $\mathbf{e}=(e)$ be a subdimension vector, i.e. $0\leq e\leq d$. The quiver Grassmannian $\Gr_\mathbf{e}(M)$ is an ordinary Grassmannian of vector subspaces of $M$:  
$$
\Gr_\mathbf{e}(M)=\{W\subseteq \CC^d|\,\textrm{dim}\,W=e\}.
$$
Choosing a basis $\{w_1,\cdots, w_e\}$ of a point $W\in\Gr_\mathbf{e}(M)$ defines a $e\times d$-matrix 
$$
A_W=\left(\begin{array}{c}w_1\\\hline w_2\\\hline\vdots\\\hline w_e\end{array}\right)
$$ of maximal rank $e$ whose rows are the elements of the chosen basis. On the other hand every $e\times d$-matrix $A$ of maximal rank $e$ defines a point $W_A\in \Gr_\mathbf{e}(M)$ which is the span of the rows of $A$. This defines a surjective map 
$$
\xymatrix{
\mathrm{Mat}_{e\times d}^{\mathrm{max}}\ar@{->>}[r]&\Gr_\mathbf{e}(M):A\ar@{|->}[r]&W_A
}
$$
where $\mathrm{Mat}_{e\times d}^{\mathrm{max}}$ is the open subset of maximal rank $e\times d$--matrices. The group $\mathrm{GL}_e$ acts freely  on $\mathrm{Mat}_{e\times d}$ by left multiplication and does not change the span of the rows. The Grassmannian is hence a geometric quotient 
$$
\Gr_\mathbf{e}(\CC^d)\simeq \mathrm{Mat}_{e\times d}^{\mathrm{max}}/\mathrm{GL}_e.
$$
The quotient map $\xymatrix{
\mathrm{Mat}_{e\times d}^{\mathrm{max}}\ar@{->>}[r]&\Gr_\mathbf{e}(M)
}
$
can be (locally) trivialized as follows: given $A\in \mathrm{Mat}_{e\times d}^{\mathrm{max}}$ there exists column indices $1\leq j_1<j_2<\cdots<j_e\leq d$ such that the $e\times e$-submatrix $A^J$ supported on the columns $J=(j_1,\cdots, j_e)$ of $A$ is invertible. Define
$$
\Delta_J(A):=\mathrm{det}(A^J)\neq0.
$$
For every $J=(j_1<\cdots< j_e)$ define 
$$
\tilde{\mathcal{U}}_J=\left\{A\in\mathrm{Mat}_{e\times d}^{\mathrm{max}}|\,\Delta_J(A)\neq0\right\}\subset \mathrm{Mat}_{e\times d}.
$$
This is an open subset. Given $A\in U_J$ we can multiply on the left by the inverse of $A^J$ and we get a matrix $\overline{A}$ such that $\overline{A}^J=\mathbf{1}_e$. We hence see that the restriction of the quotient map to $\tilde{\mathcal{U}}_J$ provides a trivial quotient 
$$
\xymatrix{
\tilde{\mathcal{U}}_J\ar@{->>}[r]&\mathcal{U}_J=\left\{W_A|\,A^J=\mathbf{1}_e\right\}\simeq\mathbb{A}^{e(d-e)}}: A\mapsto W_A
$$
so that $\tilde{\mathcal{U}}_J\simeq \mathcal{U}_J\times\mathrm{GL}_e$. 
This provides  $\Gr_\mathbf{e}(M)$ with  the structure of an $e(d-e)$--manifold. The Grasssmannian is covered by the affine spaces $\mathcal{U}_J$: 
$$
\Gr_\mathbf{e}(M)=\bigcup_J\,\mathcal{U}_J
$$
an hence $\{\mathcal{U}_J\}$ is an affine covering of the Grassmannian, called the standard affine covering. 

\begin{example} The standard affine covering for $\Gr_2(\CC^4)$ is formed by
$$
\begin{array}{cc}
\mathcal{U}_{(1,2)}=\left\{\left\langle\begin{array}{cccc}1&0&\ast&\ast\\0&1&\ast&\ast\end{array}
\right\rangle\right\},
&
\mathcal{U}_{(1,3)}=\left\{\left\langle\begin{array}{cccc}1&\ast&0&\ast\\0&\ast&1&\ast\end{array}
\right\rangle\right\},
\\
&\\
\mathcal{U}_{(1,4)}=\left\{\left\langle\begin{array}{cccc}1&\ast&\ast&0\\0&\ast&\ast&1\end{array}
\right\rangle\right\},
&
\mathcal{U}_{(2,3)}=\left\{\left\langle\begin{array}{cccc}\ast&1&0&\ast\\\ast&0&1&\ast\end{array}
\right\rangle\right\},
\\
&\\
\mathcal{U}_{(2,4)}=\left\{\left\langle\begin{array}{cccc}\ast&1&\ast&0\\\ast&0&\ast&1\end{array}
\right\rangle\right\},
&
\mathcal{U}_{(3,4)}=\left\{\left\langle\begin{array}{cccc}\ast&\ast&1&0\\\ast&\ast&0&1\end{array}
\right\rangle\right\},
\end{array}
$$
where $\ast$ denotes an arbitrary complex number and $\langle A\rangle$ denotes the span of the rows of $A$. 
\end{example}
It is well-known (see e.g. \cite[Example~6.6]{Harris}) that the Grassmannian $\Gr_e(\CC^d)$ is a smooth, projective and irreducible algebraic variety of dimension 
\begin{equation}\label{Eq:DimGrass}
\textrm{dim}\,\Gr_e(\CC^d)=e(d-e).
\end{equation}
The tangent space at $W\in\Gr_e(\CC^d)$ is 
$$
T_W(\Gr_e(\CC^d))\simeq\Hom_\CC(W,\CC^d/W).
$$
We notice that $M=\CC^d$ is a rigid $Q$--representation (cf. Proposition~\ref{Prop:RigidQGSmooth}). 

We now highlight another fundamental property of $\Gr_e(\CC^d)$: it admits a \emph{cellular decomposition}. We say that two matrices $A$ and $B$ of the same size are row-equivalent if there exists an invertible matrix $C$ (of the appropriate size) such that $B=CA$. Since invertible matrices are products of elementary matrices, we see that this happens if and only if $A$ can be transformed into $B$ via Gaussian elimination. It is not hard to prove that every matrix $A\in \mathrm{Mat}_{e\times d}$ is row equivalent to a matrix, denoted rref(A), with the following properties: 1) The zero rows are at the bottom; 2) The pivot of every (non-zero) row is one; 3) the columns containg the pivots have only the pivots as non-zero elements; 4) the pivot of the $i$-th row is on the left of the pivot of the $(i+1)$-th row. The matrix rref(A) is called the row reduced echelon form of $A$ (and this explains the notation). The columns of rref(A) containing the pivots are called dominants. We hence get a partition (i.e. a disjoint union)
\begin{equation}\label{Eq:CellDedGrass}
\Gr_e(\CC^d)=\amalg_J \mathcal{C}_J
\end{equation}
where $J=(1\leq j_1<\cdots<j_e\leq d)$ is an $e$-set of column indices and $\mathcal{C}_J$ has the following equivalent definitions
$$
\begin{array}{rcl}
\mathcal{C}_J&=&\left\{W_A|\,\textrm{rref(A) has dominant columns J}\right\}\\
&=&\left\{W\subseteq_\CC\CC^d|\,\textrm{dim}\,W_k=\left\{\begin{array}{ll}\textrm{dim}\,W_{k-1}&\textrm{ if }k\not\in I\\\textrm{dim}\,W_{k-1}+1&\textrm{ if }k\in J\end{array}\right.\right\}
\end{array}
$$
where $e_i=(0,\cdots,0,1,0,\cdots,0)$ is the $i$--th row of the identity matrix $\mathbf{1}_d$ and $W_k:=W\cap\langle e_k,\cdots, e_n\rangle$ (notice that we identify $\CC^d$ with the row matrices). 
It is clear that $\mathcal{C}_J$ is an affine space (i.e. a \emph{cell}). To make this more explicit we consider the subgroup $\textrm{U}\subset \textrm{GL}_d$ of unipotent upper triangular $d\times d$ matrices (i.e. its elements are upper triangular matrices with $1$ on the diagonal). This group acts on $\mathrm{Mat}_{e\times d}^{\mathrm{max}}$ to the right, i.e. it acts on the columns. We immediately get
$$
\mathcal{C}_J=\langle e_j|\, j\in J\rangle \textrm{U}. 
$$
This has the following interesting consequence: the closure of  a cell is a union of cells of smaller dimension
\begin{equation}\label{Eq:CellDecGrassClosure}
\overline{\mathcal{C}_J}=\amalg_I \mathcal{C}_I
\end{equation}
where $I$ varies on a subset of index-sets that can be explicitely described. Conditions \eqref{Eq:CellDedGrass} and \eqref{Eq:CellDecGrassClosure} imply at once the following properties of $X=\Gr_e(M)$: 
\begin{enumerate}
\item $H_i(X)=0$ if $i$ is odd, and it is torsion-free if $i$ is even.
\item The cycle map $A_{\bullet}(X)\rightarrow H_\bullet(X)$ is an isomorphism.
\end{enumerate}
Varieties having those two properties are said to have property~(S). 
The cellular decomposition \eqref{Eq:CellDedGrass} depends on the choice of an ordering of the standard basis vectors of $\CC^d$. It is sometimes called the standard cellular decomposition of the Grassmannian. 
\begin{example} The standard cells of $\Gr_2(\CC^4)$ are
$$
\begin{array}{cc}
\mathcal{C}_{(1,2)}=\left\{\left\langle\begin{array}{cccc}1&0&\ast&\ast\\0&1&\ast&\ast\end{array}
\right\rangle\right\},
&
\mathcal{C}_{(1,3)}=\left\{\left\langle\begin{array}{cccc}1&\ast&0&\ast\\0&0&1&\ast\end{array}
\right\rangle\right\},
\\
&\\
\mathcal{C}_{(1,4)}=\left\{\left\langle\begin{array}{cccc}1&\ast&\ast&0\\0&0&0&1\end{array}
\right\rangle\right\},
&
\mathcal{C}_{(2,3)}=\left\{\left\langle\begin{array}{cccc}0&1&0&\ast\\0&0&1&\ast\end{array}
\right\rangle\right\},
\\
&\\
\mathcal{C}_{(2,4)}=\left\{\left\langle\begin{array}{cccc}0&1&\ast&0\\0&0&0&1\end{array}
\right\rangle\right\},
&
\mathcal{C}_{(3,4)}=\left\{\left\langle\begin{array}{cccc}0&0&1&0\\0&0&0&1\end{array}
\right\rangle\right\},
\end{array}
$$
The Hasse diagram of the closure relation is the following:
$$
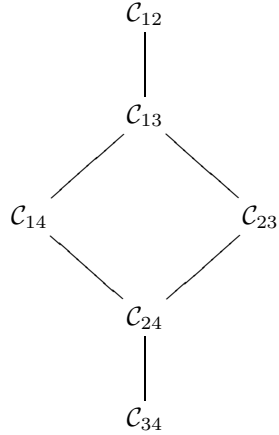
\xymatrix{
&\mathcal{C}_{12}\ar@{-}[d]&\\
&\mathcal{C}_{13}\ar@{-}[dl]\ar@{-}[dr]&\\
\mathcal{C}_{14}\ar@{-}[dr]&&\mathcal{C}_{23}\ar@{-}[dl]\\
&\mathcal{C}_{24}\ar@{-}[d]&\\
&\mathcal{C}_{34}&
}
$$
\end{example}
Another useful way to describe the cells is via torus action. Let us consider the following action of the one-dimensional torus $T=\CC^\ast$ on $\Gr_e(M)$: for every $\lambda\in T$ we rescal the standard basis vectors as $\lambda\cdot e_i=\lambda^{i-1}e_i$. This defines a linear automorphism of the vector space $K^d$ and hence descends to an action on $\Gr_e(M)$. Notice that this action depends on the ordering of the standard basis. It is hence immediate to see that 
$$
\mathcal{C}_J=\{W\in\Gr_e(\CC^d)|\,\lim_{\lambda\rightarrow 0}\lambda\cdot W=\langle e_j|\,j\in J\rangle\}.
$$ 
For example, $\lim_{\lambda\rightarrow 0}\lambda\cdot\langle e_1+e_2\rangle=\lim_{\lambda\rightarrow0}\langle e_1+\lambda e_2\rangle=\langle e_1\rangle$. 

\subsection{Example 2: The complete flag variety}\label{Sec:FlagVar}
The complete flag variety is 
$$
\mathcal{F}l_{n+1}=\left\{U_1\subset U_2\subset\cdots\subset U_n\subset\CC^{n+1}|\,\textrm{dim}\,U_i=i\right\}.
$$
This is naturally a quiver Grassmannian: Let 
\begin{equation}\label{Eq:TypeAQuiver}
\xymatrix{
Q:1\ar[r]&2\ar[r]&\cdots\ar[r]&n-1\ar[r]&n}
\end{equation}
be the equioriented quiver of type $A_n$; let 
$$
\xymatrix{
M=P_1^{n+1}:\CC^{n+1}\ar^(.7){\mathbf{1}_{n+1}}[r]&\CC^{n+1}\ar^{\mathbf{1}_{n+1}}[r]&\cdots\ar[r]&\CC^{n+1}\ar^{\mathbf{1}_{n+1}}[r]&\CC^{n+1}}
$$
and let $\mathbf{e}=(1,2,3,\cdots, n)$. Then 
$$
\mathcal{F}l_{n+1}\simeq \Gr_\mathbf{e}(M). 
$$
The complete flag variety is smooth, irreducibile of minimal dimension
$$
\textrm{dim}\,\mathcal{F}l_{n+1}=\frac{n(n+1)}{2}=\langle\mathbf{e},\mathbf{dim}\,M-\mathbf{e}\rangle.
$$
We notice that $M$ is indeed a rigid $Q$--representation (cf. Proposition~\ref{Prop:RigidQGSmooth})
\subsection{Example 3: The complete degenerate flag variety}\label{Sec:DegFlagVar}
Let $V=\CC^{n+1}$ with standard basis $\{e_1,\cdots, e_{n+1}\}$. For every $k=1, \cdots, n+1$ consider the projection along $e_k$:
$$
pr_k:V\rightarrow V:\,\sum_i x_i e_i\mapsto\sum_{i\neq k}x_ie_i. 
$$
Motivated by the study of abelian degenerations of simple Lie algebras, E.~Feigin introduced the projective variety $\mathcal{F}l_{n+1}^a$ called the  (complete) $sl_n$-degenerate flag variety \cite{F1,F2,FF}. He showed that it has a realization in terms of linear algebra  as follows
$$
\mathcal{F}l_{n+1}^a\simeq\{(U_1,\cdots, U_n)\in\prod_{k=1}^n\Gr_k(V)|\, pr_{k+1}(U_k)\subseteq U_{k+1}\}
$$
and he proved that this  projective variety has marvellous properties: it is a (typically) singular, irreducible projective variety of dimension 
$$
\textrm{dim}\,\mathcal{F}l_{n+1}^a=\textrm{dim}\,\mathcal{F}l_{n+1}=\frac{n(n+1)}{2}
$$
which is a flat degeneration of the complete flag variety $\mathcal{F}l_{n+1}$; moreover, it is a normal, locally complete intersection variety which admits a cellular decomposition. Let $Q$ be the equioriented quiver \eqref{Eq:TypeAQuiver} of type $A_n$. Let $A=KQ$ be its path algebra. As a $Q$--representation $A=\oplus_{i=1}^nP_i$ is the direct sum of all the indecomposable projectives and hence $\mathbf{dim}\,A=(1,2,\cdots, n)$. Dually, $DA=\oplus_{k=1}^nI_k$ is the sum of all the indecomposable injectives and $\mathbf{dim}\,DA=(n, n-1, n-2,\cdots, 1)$. We notice that 
$$
\mathbf{dim}\,(A\oplus DA)=(n+1,n+1,\cdots, n+1)=\mathbf{dim}\,P_1^{n+1}.
$$
From the definition it is immediate to check that 
$$
\mathcal{F}l_{n+1}^a= \Gr_{\mathbf{dim}\,A}(A\oplus DA). 
$$

\subsection{Example 4: Singular quiver Grassmannians}\label{Ex4}
We give two easy examples of singular quiver Grassmannians. 
Let $Q:1\rightarrow 2$ be an $A_2$ quiver. 

The easiest example of a non-smooth quiver Grassmannian is the following:  let 
$$
\xymatrix@C=50pt{
M=\CC^2\ar^(.4){\left(\begin{array}{cc}1&0\\0&0\end{array}\right)}[r]&\CC^2\simeq S_1\oplus P_1\oplus S_2
}
$$
and let $\mathbf{e}=(1,1)$. Then $\Gr_\mathbf{e}(M)$ is the union of two $\PP^1$'s crossing in one point. Thus $\Gr_\mathbf{e}(M)$ is a connected, equidimensional curve of dimension one with two irreducible components and one singular point. 

An easy example of a singular non-equidimensional quiver Grassmannian is the following: let 
$$
\xymatrix@C=50pt{
M=\CC^2\ar^(.4){\left(\begin{array}{cc}1&0\\0&0\\0&0\end{array}\right)}[r]&\CC^3\simeq S_1\oplus P_1\oplus S_2^2
}
$$
and let $\mathbf{e}=(1,1)$. Then $\Gr_\mathbf{e}(M)$ is the union of a $\PP^2$ and a $\PP^1$ crossing in one point. Thus $\Gr_\mathbf{e}(M)$ is a connected projective variety of dimension two with two irreducible components (one of dimension $1$ and one of dimension $2$) and one singular point. 
\subsection{Example 5: A non--connected quiver Grassmannian}
Let us give an easy example of a non--connected quiver Grassmannian. Let $\xymatrix{Q:1\ar[r]\ar@<1ex>[r]&2}$ be the Kronecker quiver. Let $A$ be a $2\times 2$ complex matrix with distinct eigenvalues. Let us consider the $Q$-representation
$$
\xymatrix{
M=\CC^2\ar@<1ex>^(.6){\mathbf{1}_2}[r]\ar_(.6){A}[r]&\CC^2}
$$
and let $\mathbf{e}=(1,1)$. Then $\Gr_\mathbf{e}(M)$ consists of two distinct points (the two eigenspaces). The tangent space at those two points is zero dimensional and hence $\Gr_\mathbf{e}(M)$ is a reduced projective variety of dimension $0$ with two connected components. 
\subsection{Example 6: a smooth quiver Grassmannian with negative Euler characteristic}
We borrow this example from Derksen-Weyman-Zelevinsky's paper \cite[Example~3.6]{DWZ2}. Let $\xymatrix{Q:1\ar[r]\ar@<0.5ex>[r]\ar@<1ex>[r]\ar@<1.5ex>[r]&2}$ be the $4$-Kronecker quiver. Let $\mathbf{d}=(3,4)$ and let $\mathbf{e}=(1,3)$. We notice that $\langle\mathbf{e,d-e}\rangle=1$. It is easy to construct representations $N\in R_\mathbf{e}$ and $R\in R_\mathbf{d-e}$ such that $[N,R]=1=\langle\mathbf{e,d-e}\rangle$ and hence $[N,R]^1=0$. By Theorem~\ref{Thm:Schofield}, the universal quiver Grassmannian $\Gr_\mathbf{e}^Q(\mathbf{d})\rightarrow R_\mathbf{d}(Q)$ is surjective. Thus, there exists an open and dense subset $U$ of $R_\mathbf{d}$ such that the fiber $\Gr_\mathbf{e}(M)$ over a point $M\in U$ is smooth of minimal dimension $\langle\mathbf{e,d-e}\rangle=1>0$.  Let $X=\Gr_\mathbf{e}(M)$ be a generic fiber (i.e. $M\in U$). Then $X$ is a smooth curve  of degree $4$ and genus $3$. It hence follows that its Euler characteristic is negative: 
$
\chi(\Gr_\mathbf{e}(M))=-4<0. 
$
Notice that since $\langle\mathbf{d,d}\rangle=-11<0$, $U$ does not contain an open orbit, i.e. there is no rigid representation of dimension vector $\mathbf{d}$. 
\subsection{Every projective variety is a quiver Grassmannian}\label{Sec:EveryProj}
It is well known that every projective variety can be realized as the intersection of a Veronese variety with a linear space (see \cite[Example~2.9]{Harris}). Markus Reineke noticed that this construction can be used in a straightforward way to realize every projective variety as a quiver Grassmannian \cite{REveryProj}. In the next section~\ref{Sec:LeBruynExample} it is shown how this works in an example.  In this costruction, the quiver and the quiver representation depends on the chosen projective variety. Surprisingly, Ringel proved that the choice of the quiver does not depend on the variety, as long as the quiver is \emph{wild}.  
\begin{theorem}\cite{RingEveryProj}
Let $X$ be a projective variety and let $Q$ be a \emph{wild} quiver. Then there exists a $Q$--representation $M$ and a dimension vector $\mathbf{e}$ such that $$X\simeq \Gr_\mathbf{e}(M).$$
\end{theorem}

\subsection{Realization of an elliptic curve as a quiver Grassmannian}\label{Sec:LeBruynExample}
Following Reineke \cite{REveryProj}, in this section we realize an elliptic curve as a quiver Grassmannian. This example appeared in the blog of Le Bruyn \cite{LeBruynBlog}. Let us consider the elliptic curve
$$
\mathcal{E}=\{[x:y:z]\in\PP^2|\, y^2z=x^3+z^3\}.
$$
We fix the complex vector space $V=\CC^3$ with standard basis $\{e_1,e_2,e_3\}$ and its linear dual $V^\ast$ with dual basis $\{e_1^\ast,e_2^\ast,e_3^\ast\}$. The first thing we need to do is to linearize the equation defining the elliptic curve: we do this using the Veronese embedding 
$
j:V\rightarrow \textrm{Sym}^3(V):\; v\mapsto v\otimes v\otimes v.
$
Consider the linear form
$\varphi= (e_2^\ast)^2e_3^\ast-(e_1^\ast)^3-(e_3^\ast)^3:\textrm{Sym}^3(V)\rightarrow \CC$
then 
\begin{equation}\label{Eq:LeBruyn1}
\mathcal{E}=\{[v]\in \Gr_1(V)|\, \varphi(j(v))=0\}\simeq \{[w]\in \Gr_1(\textrm{Sym}^3 V)|\, \varphi(w)=0\}
\end{equation}
where $[v]$ denotes the line generated by a non-zero vector $v$ and the isomorphism is induced by the embedding $j$.  

Next, we need to describe the image of the Veronese embedding. Recall that a tensor $\omega\in V^{\otimes n}$ of the form $\omega=v^{\otimes n}$ is called decomposable. Thus, the image of $j$ consists of the decomposable tensors of $V^{\otimes 3}$. 
For $1\leq i\leq j\leq k\leq 3$ we define 
$$
\begin{array}{cc}
e_ie_j:=e_i\otimes e_j+e_j\otimes e_i,&
e_ie_je_k:=\sum_{\sigma\in S_3}e_{\sigma(i)}\otimes e_{\sigma(j)}\otimes e_{\sigma(k)}.
\end{array}
$$
The sets $\{e_ie_j\}$ and $\{e_ie_je_k\}$ form a basis of  $\textrm{Sym}^2(V)$ and  of $\textrm{Sym}^3(V)$, respectively. It is easy to describe the decomposable vectors of $V\otimes V$: 
consider the canonical isomorphism 
$$
\zeta_?:V\otimes V\simeq \Hom(V^\ast,V): \omega=v\otimes w\mapsto \zeta_\omega=(u\mapsto u(v)w)
$$
defined on the decomposable tensors and extended by linearity. It is immediate to verify that  $\omega$ is decomposable  if and only if $\textrm{rk}(\zeta_\omega)\leq 1$. 
We use this criterion to detect the decomposable tensors of $V\otimes V\otimes V$: 
We consider the linear map 
$$
\iota:\textrm{Sym}^3(V)\rightarrow V\otimes\textrm{Sym}^2(V)
$$
defined on the basis elements by
$
 e_ie_je_k\mapsto e_i\otimes e_je_k+e_j\otimes e_ie_k+e_k\otimes e_ie_j.
$
This is the injective linear map compatible with the inclusions $\textrm{Sym}^3(V)\subset V^{\otimes 3}$ and $V\otimes\textrm{Sym}^2(V)\subset V^{\otimes 3}$. In particular, it sends a decomposable vector to a decomposable vector: Indeed 
$$
\iota(\sum_{i\leq j\leq k}\alpha_i\alpha_j\alpha_k\,e_ie_je_k)=(\sum_i\alpha_ie_i)\otimes (\sum_{j\leq k}\alpha_j\alpha_ke_je_k).
$$
The standard basis $\{e_1,e_2,e_3\}$ of $V$ determines an isomorphism of vector spaces $V\otimes\textrm{Sym}^2(V)\simeq \textrm{Sym}^2(V)\oplus \textrm{Sym}^2(V)\oplus \textrm{Sym}^2(V)$. Let 
$$
\pi_1,\pi_2,\pi_3:V\otimes\textrm{Sym}^2(V)\rightarrow \textrm{Sym}^2(V)
$$
be the projections onto the three factors, respectively. Thus, by definition, a vector $\omega$ of $V\otimes\textrm{Sym}^2(V)$ is written as 
$
\omega=e_1\otimes \pi_1(\omega)+e_2\otimes \pi_2(\omega)+e_3\otimes \pi_3(\omega).
$
Let
$$
\xymatrix@C=40pt{
\psi_1,\psi_2,\psi_3:\textrm{Sym}^3(V)\ar^(.6)\iota[r]&V\otimes\textrm{Sym}^2(V)\ar^{\pi_1,\pi_2,\pi_3}[r]& \textrm{Sym}^2(V)
}
$$
be the composite maps $\psi_i=\pi_i\circ \iota$. From the criterion above we see that $\omega$ is decomposable if and only if $\pi_1(\omega)$, $\pi_2(\omega)$ and $\pi_3(\omega)$ are all contained in a same line. We hence have: 
\begin{equation}\label{Eq:LeBruyn2}
\xymatrix@C=12pt{
t\in\textrm{Im}\,j\ar@{<=>}[r]&\iota(t)\textrm{ is decomposable}\ar@{<=>}[r]&\textrm{dim Span}\{\psi_1(t), \psi_2(t), \psi_3(t)\}\leq 1.
}
\end{equation}

We consider the quiver 
$
\xymatrix{
\bullet &\ar[l]\ar@<1ex>[r]\ar@<-1ex>[r]\ar[r]\bullet&\bullet
}
$ 
and its representation
$$
M=\xymatrix@C=40pt{
\CC &\ar_(.6)\varphi[l]\ar@<1ex>^{\psi_1,\psi_2,\psi_3}[r]\ar@<-1ex>[r]\ar[r]\textrm{Sym}^3(V)&\textrm{Sym}^2(V)
}
$$ 
Then, putting together \eqref{Eq:LeBruyn1} and \eqref{Eq:LeBruyn2}, we have  
$
\mathcal{E}\simeq \Gr_\mathbf{e}(M)
$
where $\mathbf{e}=(0,1,1)$.

\section{Quiver Grassmannians of type~\textrm{A}}\label{Sec:TypeA}
In sections \eqref{Sec:FlagVar} and \eqref{Sec:DegFlagVar} we saw that the complete flag variety and the corresponding degenerate flag variety are quiver Grassmannians attached to representations of the equioriented quiver of type $A_n$. In this section we study some general properties of quiver Grassmannians of this sort, that we briefly call ``of type $A$''.

Let $Q$ be the equioriented quiver \eqref{Eq:TypeAQuiver} of type $A_n$. A $Q$--representation  $M=((M_i)_{i=1}^n, (f_i)_{i=1}^{n-1})$ is
$$
\xymatrix{
M:M_1\ar^{f_1}[r]&M_2\ar^{f_2}[r]&\cdots\ar^{f_{n-2}}[r]&M_{n-1}\ar^{f_{n-1}}[r]&n}.
$$
It is an interesting problem of linear algebra to find the normal form of a collection of linear maps $(f_1,\cdots, f_{n-1})$ by base change. It turns out that the indecomposable $Q$--representations are thin and they are supported on connected subgraphs of $Q$. Given $1\leq i\leq j\leq n$ we denote by $U_{i,j}$ the indecomposable supported on the interval $[i,j]$. The projectives are $P_i=U_{i,n}$ and the injectives are $I_k=U_{1,k}$. The AR-translate $\tau U_{i,j}$ of $U_{i,j}$ is $\tau U_{i,j}=U_{i+1,j+1}$ and the AR-quiver is the following (for $n=4$): 
$$
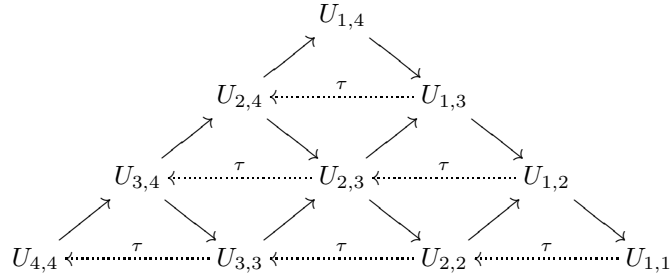
\xymatrix@R=15pt@C=15pt{
&&&U_{1,4}\ar[dr]&&&\\
&&U_{2,4}\ar[ur]\ar[dr]&&U_{1,3}\ar[dr]\ar@{..>}_\tau[ll]&&\\
&U_{3,4}\ar[ur]\ar[dr]&&U_{2,3}\ar[ur]\ar[dr]\ar@{..>}_\tau[ll]&&U_{1,2}\ar[dr]\ar@{..>}_\tau[ll]&\\
U_{4,4}\ar[ur]&&U_{3,3}\ar[ur]\ar@{..>}_\tau[ll]&&U_{2,2}\ar[ur]\ar@{..>}_\tau[ll]&&U_{1,1}\ar@{..>}_\tau[ll]\\
}
$$
The space of homomorphisms (and of extensions) between them are at most one-dimensional and are given as follows
\begin{equation}
[U_{ij}, U_{k\ell}]=\left\{\begin{array}{cc}1&\textrm{ if }k\leq i\leq \ell\leq j\\0&\textrm{otherwise}\end{array}\right.
\end{equation}
\begin{equation}\label{Eq:ExtSpaceTypeA}
[U_{k\ell},U_{ij}]^1=[U_{ij}, U_{k+1\ell+1}]=\left\{\begin{array}{cc}1&\textrm{ if }k+1\leq i\leq \ell+1\leq j\\0&\textrm{otherwise}\end{array}\right.
\end{equation}

We can order the indecomposable $Q$--representations as $M(1)<M(2)<\cdots<M(N)$ (where $N=\frac{n(n+1)}{2}$) so that 
\begin{equation}\label{Eq:OrderTypeA}
\xymatrix{
M(i)<M(\ell)\ar@{=>}[r]&\Ext^1(M(i),M(\ell))=0.
}
\end{equation}
As shown in \cite[Remark~7]{CFFFR}, a natural choice is the following (for $n=4$)
$$
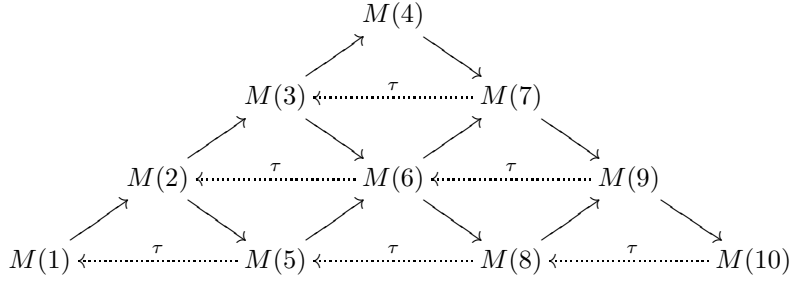
\xymatrix@R=15pt@C=15pt{
&&&M(4)\ar[dr]&&&\\
&&M(3)\ar[ur]\ar[dr]&&M(7)\ar[dr]\ar@{..>}_\tau[ll]&&\\
&M(2)\ar[ur]\ar[dr]&&M(6)\ar[ur]\ar[dr]\ar@{..>}_\tau[ll]&&M(9)\ar[dr]\ar@{..>}_\tau[ll]&\\
M(1)\ar[ur]&&M(5)\ar[ur]\ar@{..>}_\tau[ll]&&M(8)\ar[ur]\ar@{..>}_\tau[ll]&&M(10)\ar@{..>}_\tau[ll]\\
}
$$
To a $Q$-representation $M=((M_i),(f_i))$ we associate the sequence of non-negative integers $\mathbf{r}^M=(r_{i,j}^M|\, 1\leq j\leq n)$ given by the ranks of the composite linear maps $f_i$'s, i.e.  $r_{i,j}^M=rk(f_{j-1}\circ\cdots\circ f_i)$ for $i<j$ and $r_{i,i}^M=\textrm{dim}\,M_i$.  If $M$ decomposes as direct sum of indecomposables as $M=\oplus_{i,j}U_{i,j}^{m_{i,j}}$, the relation between the rank tuple $\mathbf{r}^M$ and the tuple of multiplicities $(m_{i,j})$ is given by 
$$
\begin{array}{cc}
r_{i,j}^M=\sum_{k\leq i\leq j\leq \ell}m_{k,\ell},&m_{ij}^M=r_{i,j}^M-r_{i-1,j}^M-r_{i,j+1}^M+r_{i-1,j+1}^M.
\end{array}
$$
In particular, the sequence $(r_{i,j}^M)$ satisfies the inequalities
$$
r_{i,j}^M+r_{i-1,j+1}^M\geq r_{i,j+1}^M+ r_{i-1,j}^M.
$$
Conversely, let $\mathbf{r}=(r_{i,j}|\, 1\leq i\leq j\leq n)$ be a sequence of non--negative integers which fulfill the inequalities 
\begin{equation}\label{Eq:IneqTypeA}
r_{i,j}+r_{i-1,j+1}\geq r_{i,j+1}+r_{i-1,j}
\end{equation}
for all $1\leq i\leq j\leq n$, with the convention $r_{i,j}=0$ if $i=0$ or $j=n+1$. An easy induction shows that \eqref{Eq:IneqTypeA} are equivalent to 
\begin{equation}\label{Eq:IneqTypeA2}
r_{i,k}+r_{j,\ell}\geq r_{j,\ell}+r_{i,k}
\end{equation}
for every $1\leq i\leq j\leq k\leq \ell\leq n$. Then one can easily contruct a $Q$-representation $M$ such that $\mathbf{r}^M=\mathbf{r}$. 
A sequence of non-negative integers $\mathbf{r}=(r_{i,j})$ satisfying the inequalities \eqref{Eq:IneqTypeA2} is called a rank sequence. By the above, $M\simeq N$ if and only if $\mathbf{r}^M=\mathbf{r}^N$, i.e. the isoclasses of $Q$--representations are parametrized by rank sequences.  Moreover it is proved in \cite{AdF1} that $M\leq_{deg}N$ if and only if $r_{ii}^M=r_{ii}^N$ and $r_{ij}^M\geq r_{ij}^N$ for every $i,j$; in this case we briefly write $\mathbf{r}^M\geq \mathbf{r}^N$.\subsection{Cellular decomposition}

Let $M=(M_i,f_i)$ be a $Q$--representation. Decompose $M=\oplus_{k=1}^s M(k)$ as a direct sum of its indecomposable direct summands, so that $[M(i),M(j)]^1=0$ for $i<j$. 

For every $\lambda\in\CC^\ast$ consider the automorphism $f_\lambda:M\rightarrow M$ of $M$ which rescales the $M(k)$'s as follows
$$
f_\lambda(m)=\lambda^{k-1}m\;\;\;\forall m\in M(k).
$$
This gives an action 
$$
T\times\Gr_\mathbf{e}(M)\rightarrow\Gr_\mathbf{e}(M):\;(\lambda,N)\mapsto \lambda\cdot N
$$ 
of the one-dimensional torus $T=\CC^\ast$ on every quiver Grassmannian associated with $M$. This action has finitely many  $T$--fixed points which are coordinate subrepresentations:
\begin{equation}\label{Eq:TorusFixedPointsTypeA}
\Gr_\mathbf{e}(M)^T=\prod_{\mathbf{f}_1+\cdots+\mathbf{f}_s=\mathbf{e}}\Gr_{\mathbf{f}_1}(M(1))\times\cdots\times \Gr_{\mathbf{f}_s}(M(s))
\end{equation}
Notice that since the $M(k)$'s are thin, every non-empty quiver Grassmannian associated to them is a point and $\Gr_\mathbf{e}(M)^T$ is a finite collection of points.  For a $T$-fixed point $L$ consider the attracting space 
$$
\mathcal{C}_L=\{N\in\Gr_\mathbf{e}(M)|\,\lim_{\lambda\rightarrow 0}\lambda.N=L\}.
$$
We can represent the representation $M$ as a collection of strings, ordered from top to bottom, according to \eqref{Eq:OrderTypeA}. This collection of strings is called the coefficient quiver of $M$. Given a $T$-fixed point $L\in \Gr_\mathbf{e}(M)^T$ we colour black the vertices corresponding to $L$ and to compute the dimension of $\mathcal{C}_L$ we only need to count how many white vertices there are below each black \emph{source} of $L$. For example, if $n=3$ and $M=A\oplus DA$, $\mathbf{e}=\mathbf{dim}\,A=(1,2,3)$ the following is coefficient quiver of $M$ together with a $T$-fixed point $L$ highlighted by black vertices
$$
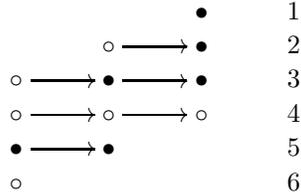
\xymatrix@R=0.5pt{
&&\bullet&1\\
&\circ\ar[r]&\bullet&2\\
\circ\ar[r]&\bullet\ar[r]&\bullet&3\\
\circ\ar[r]&\circ\ar[r]&\circ&4\\
\bullet\ar[r]&\bullet&&5\\
\circ&&&6}
$$
Then the dimension of $\mathcal{C}_L$ is $4$ (see \cite[Sec.~6.4]{CFFFR}). 
By elementary linear algebra techniques one can prove the following result. 
\begin{theorem}\label{Thm:CellDecTypeA}\cite[Thm~12]{CFFFR}
For every $T$--fixed point $L$, the attrcting set $\mathcal{C}_L$ is an affine space and the quiver Grassmannian $\Gr_\mathbf{e}(M)$ admits a cellular decomposition
$$
\Gr_\mathbf{e}(M)=\coprod_{L\in\Gr_\mathbf{e}(M)^T}\mathcal{C}_L.
$$
Moreover the points of every cell $\mathcal{C}_L$ are isomorphic to $L$ and hence every iso-stratum $\mathcal{S}_L$  decomposes as union of cells. 
\end{theorem}
The choice of the ordering of the indecomposable direct summands $M(i)$'s of $M$ is necessary to make theorem~\ref{Thm:CellDecTypeA} working. Indeed, consider the case $n=2$ and the representation  $M=S_1\oplus P_1\oplus S_2$ of example~4 with the ordering given by 
$$
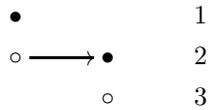
\xymatrix@R=3pt{
\bullet&&1\\
\circ\ar[r]&\bullet&2\\
&\circ&3
}
$$
and let $L$ be the highlighted point of $\Gr_{(1,1)}(M)$. One can easily verify that $\mathcal{C}(L)\simeq \{([1:\lambda],[1:\mu])\in\PP^1\times\PP^1|\, \lambda\mu=0\}$ which is not a cell.
\subsection{Schubert quiver Grassmannians}
Let $Q$ be the equioriented quiver of type $A_n$ as in the previous section. A $Q$--representation $R=((R_i), (f_i))$ is projective if and only if every linear map $f_i:R_{i}\rightarrow R_{i+1}$ is injective. In other words, $R$ is projective if and only if $R_\bullet$ is a flag in $R_n$. The automorphism group $Aut(R)$ of $R$ consists of those $g\in GL(R_n)$ which fix the flag $R_\bullet$; thus, $Aut(R)$ is  a parabolic subgroup of $GL(R_n)$. This implies that the quiver Grassmannians associated with $R$ are Schubert varieties inside the partial flag variety $GL(R_n)/Aut(R)$. Given a $Q$-representation $M$, let us consider its minimal projective resolution
$$
\xymatrix{
0\ar[r]&P\ar^\iota[r]&R\ar^\pi[r]&M\ar[r]&0}.
$$
For example if $n=3$ and $M=A\oplus DA$ then the diagram 
$$
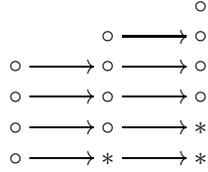
\xymatrix@R=1pt{
&&\circ\\
&\circ\ar[r]&\circ\\
\circ\ar[r]&\circ\ar[r]&\circ\\
\circ\ar[r]&\circ\ar[r]&\circ\\
\circ\ar[r]&\circ\ar[r]&\ast\\
\circ\ar[r]&\ast\ar[r]&\ast}
$$
describes the minimal projective resolution of $M$: the $\ast$ form (the coefficient-quiver of ) $P$, the whole diagram is (the coefficient-quiver of )  $R$ and the diagram without $\ast$ is (the coefficient-quiver of )  $M$. 
We can use the surjective morphism $\pi:R\rightarrow M$ to embed  a quiver Grassmannian $\Gr_\mathbf{e}(M)$ associated with $M$ into the partial flag variety  $GL(R_n)/Aut(R)$ (see \cite[Prop.~2.1]{CFR4}): we define the map
$$
\zeta:\Gr_\mathbf{e}(M)\rightarrow\Gr_{\mathbf{e}+\mathbf{dim}\,P}(R):\; N\mapsto\pi^{-1}(N).
$$
This is a closed embedding and we refer to it as the standard embedding of $\Gr_\mathbf{e}(M)$ inside a partial flag variety. Let $B\subseteq Aut(R)$ be the Borel subgroup of $GL(R_n)$ contained in the parabolic subgroup $Aut(R)$. It is a natural problem to find conditions on the $Q$--representation $M$ and to the dimension vector $\mathbf{e}$ so that the standard embedding of $\Gr_\mathbf{e}(M)$ into a partial flag manifold is stable by $B$. Indeed, if this happens, then the irreducible components of $\Gr_\mathbf{e}(M)$ are Schubert varieties. For this reason, we called such varieties Schubert quiver Grassmannians. 

 In \cite{CFR4} we studied this problem and found a purely combinatorial solution. To state the result we introduced the following terminology:

\begin{definition}\label{Def:Catenoid}
We say that a $Q$--representation $M$ is a \emph{catenoid} if the indecomposable direct summands of $M$ lie in an oriented (connected) path of  the AR-quiver of $Q$. 
\end{definition}

To illustrate definition~\ref{Def:Catenoid} let us consider the case $n=4$; given a $Q$-representation $M$ we highlight with $\bullet$ the vertices of the AR-quiver corresponding to the indecomposable direct summands of $M$. Consider the following two configurations:
$$
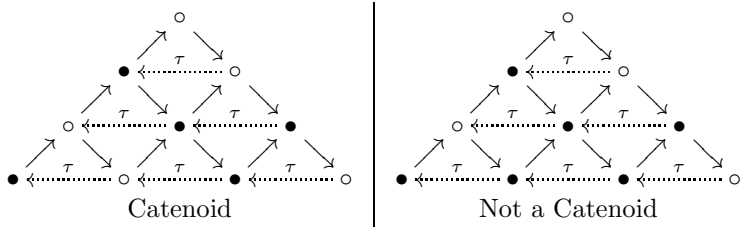
\begin{array}{c|c}
\xymatrix@R=10pt@C=10pt{
&&&\circ\ar[dr]&&&\\
&&\bullet\ar[ur]\ar[dr]&&\circ\ar[dr]\ar@{..>}_\tau[ll]&&\\
&\circ\ar[ur]\ar[dr]&&\bullet\ar[ur]\ar[dr]\ar@{..>}_\tau[ll]&&\bullet\ar[dr]\ar@{..>}_\tau[ll]&\\
\bullet\ar[ur]&&\circ\ar[ur]\ar@{..>}_\tau[ll]&&\bullet\ar[ur]\ar@{..>}_\tau[ll]&&\circ\ar@{..>}_\tau[ll]\\
}
&
\xymatrix@R=10pt@C=10pt{
&&&\circ\ar[dr]&&&\\
&&\bullet\ar[ur]\ar[dr]&&\circ\ar[dr]\ar@{..>}_\tau[ll]&&\\
&\circ\ar[ur]\ar[dr]&&\bullet\ar[ur]\ar[dr]\ar@{..>}_\tau[ll]&&\bullet\ar[dr]\ar@{..>}_\tau[ll]&\\
\bullet\ar[ur]&&\bullet\ar[ur]\ar@{..>}_\tau[ll]&&\bullet\ar[ur]\ar@{..>}_\tau[ll]&&\circ\ar@{..>}_\tau[ll]\\
}
\\\textrm{Catenoid}&\textrm{Not a Catenoid}
\end{array}
$$
Then a representation whose configuration of its indecomposables is shown on the left  is a catenoid, while the one on the right is not a catenoid. Notice that the multiplicities do not play any r\^ole.  
\begin{theorem}\label{Thm:SchubertQuivGrass}
A quiver Grassmannian $\Gr_\mathbf{e}(M)$ is a Schubert quiver Grassmannian if and only if $M$ is a catenoid. 
\end{theorem} 
Theorem~\ref{Thm:SchubertQuivGrass} has the following interesting corollary. 
\begin{corollary}\label{Cor:CL}
Degenerate flag varieties are Schubert varieties. 
\end{corollary}
\begin{proof}
The complete degenerate flag variety is $\Gr_{\mathbf{dim}\,A}(A\oplus DA)$ and the representation $M=A\oplus DA$ is a catenoid. By Theorem~\ref{Thm:SchubertQuivGrass} its irreducible components are hence Schubert varieties. Since the complete degenerate flag variety is irreducible the result follows. 
\end{proof}
Corollary~\ref{Cor:CL}  was first proved in collaboration with Martina Lanini \cite{CL} and it was the starting point for the study of Schubert quiver Grassmannians.  It holds for partial degenerate flag varieties, and  for symplectic degenerate flag varieties. For a characteristic-free approach see \cite{CLL}. 

\subsection{Linear degeneration of the complete flag variety}\label{Sec:LinDeg}
Let $Q$ be the equioriented quiver of type $A_n$ and let $A=KQ$ be its path algebra. The dimension vector $\mathbf{d}=(n+1,n+1,\cdots, n+1)$ is of special intereset. Indeed, both $P_1^{n+1}$ and $A\oplus DA$ have dimension vector equal to $\mathbf{d}$. The complete flag variety $\mathcal{F}l_{n+1}$ and the degenerate flag variety $\mathcal{F}l_{n+1}^a$ are hence fibers of the universal quiver Grassmannian 
$$
p_\mathbf{e,d}:\Gr_\mathbf{e}^Q(\mathbf{d})\rightarrow R_\mathbf{d}
$$
where $\mathbf{e}=\mathbf{dim}\,A$. Notice that $p_\mathbf{e,d}$ is surjective because the rigid representation $P_1^{n+1}$ belongs to its image. The family $p_\mathbf{e,d}$ is hence of special interest and it was studied in \cite{CFFFR}. The main result of \cite{CFFFR} describes the locus where $p_\mathbf{e,d}$ is flat with irreducible fibers, denoted with $U_{flat,irr}$ and the locus where $p_\mathbf{e,d}$ is flat, denoted with $\mathcal{U}_{flat}$. Since $\mathcal{F}l_{n+1}=p_\mathbf{e,d}^{-1}(P_1^{n+1})$ and $\mathcal{F}l_{n+1}^a=p_\mathbf{e,d}^{-1}(A\oplus DA)$ are irreducible, both $P_1^{n+1}$ and $A\oplus DA$ lie in $\mathcal{U}_{flat,irr}$. 
Let us consider the rank sequences $\mathbf{r}^0$, $\mathbf{r}^1$ and $\mathbf{r}^2$ given by
$$
\begin{array}{ccc}
r_{i,j}^0=n+1,&r_{i,j}^1=n+1-(j-i),&r_{i,j}^2=n-(j-i).
\end{array}
$$
We put $M^0=P_1^{n+1}$, $M^1=A\oplus DA$ and notice that $\mathbf{r}^0=\mathbf{r}^{M^0}$ and $\mathbf{r}^{M^1}$. The representation $M^2$ such that $\mathbf{r}^2=\mathbf{r}^{M^2}$ is 
$$
M^2=\bigoplus_{i=1}^n P_i\oplus \bigoplus_{j=1}^{n-1} I_j\oplus\bigoplus_{k=1}^nS_k=A\oplus DA/(\textrm{soc}\,DA)\oplus \textrm{soc}\,DA
$$
The following are the coefficient quivers of $M^0$, $M^1$ and $M^2$ respectively, for $n=3$
$$
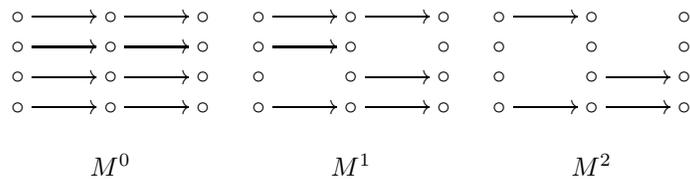
\begin{array}{ccc}
\xymatrix@R=1pt{
\circ\ar[r]&\circ\ar[r]&\circ\\
\circ\ar[r]&\circ\ar[r]&\circ\\
\circ\ar[r]&\circ\ar[r]&\circ\\
\circ\ar[r]&\circ\ar[r]&\circ}
&
\xymatrix@R=1pt{
\circ\ar[r]&\circ\ar[r]&\circ\\
\circ\ar[r]&\circ&\circ\\
\circ&\circ\ar[r]&\circ\\
\circ\ar[r]&\circ\ar[r]&\circ}
&
\xymatrix@R=1pt{
\circ\ar[r]&\circ&\circ\\
\circ&\circ&\circ\\
\circ&\circ\ar[r]&\circ\\
\circ\ar[r]&\circ\ar[r]&\circ}
\\&&\\
M^0&M^1&M^2
\end{array}
$$
The following theorem describes the loci $\mathcal{U}_{\textrm{flat,Irr}}$ and $\mathcal{U}_{\textrm{ flat}}$ defined above. 
\begin{theorem}\label{Thm:CFFFR}\cite{CFFFR}
\begin{enumerate}
\item $\mathcal{U}_{\textrm{flat,Irr}}=\{M\in R_\mathbf{d}|\, M\leq_{deg} M^1\}=\{M\in R_\mathbf{d}|\, \mathbf{r}^M\geq\mathbf{r}^1\}$.
\item $\mathcal{U}_{\textrm{flat}}=\{M\in R_\mathbf{d}|\, M\leq_{deg} M^2\}=\{M\in R_\mathbf{d}|\, \mathbf{r}^M\geq\mathbf{r}^2\}$.
\end{enumerate}
\end{theorem}
It is worth noting that the irreducible flat locus $\mathcal{U}_{\textrm{flat,Irr}}$ coincides with the normal flat locus, i.e. the locus consisting of points $M$ whose fiber $p_\mathbf{e,d}^{-1}(M)$ is normal and of the same dimension as the complete flag variety $\mathcal{F}l_{n+1}$. 

In view of Theorem~\ref{Thm:CFFFR} we call $\Gr_{\mathbf{dim}\,A}(M^2)$ the most-flat linear degeneration of the complete flag variety, or, in short, mf-linear degeneration of the flag variety. It turns out that the mf-linear degeneration has an interesting geometric structure: it is an equi-dimensional, locally complete intersection variety whose irreducible components are naturally parametrized by non-crossing arc diagrams on $n$ vertices and hence they are in number of $C_n=\frac{1}{n+1}$$\binom{2n}{n}$, the n-th Catalan number.  

The flat irreducible locus contains two Schubert varieties, namely the complete flag variety $\mathcal{F}l_{n+1}$ and the degenerate flag variety $\mathcal{F}l^a_{n+1}$. We call $\mathcal{U}_{PBW}\subset \mathcal{U}_{irr,flat}$ the locus of points whose fiber is a Schubert quiver Grassmannian. In \cite[Section~5]{CFFFR} those Schubert quiver Grassmannian are shown to be PBW-degenerations of the complete flag variety, and hence the name. Moreover a complete description of the realization as a Schubert variety is provided.

\section{Quiver Grassmannians of Dynkin type}
In section~\ref{Sec:TypeA} we saw that even quiver Grassmannians of type~A can have a complicated geometric structure. We hence cannot expect to find general results concerning quiver Grassmannians associated with an arbitrary Dynkin quiver. What we can do is to restrict our attention to particular quiver Grassmannians. We have in mind the complete degenerate flag variety $\mathcal{F}l_{n+1}\simeq \Gr_{\mathbf{dim}\,A}(P_1^{n+1})$,  and the degenerate flag variety $\mathcal{F}l^a_{n+1}\simeq \Gr_{\mathbf{dim}\,A}(A\oplus DA)$, where $A$ is the path algebra of the equioriented quiver of type $A_n$. Those varieties share many nice properties. It is hence natural to consider quiver Grassmannians of the form $\Gr_{\mathbf{dim}\,P}(P\oplus I)$ where $P$ is a projective and $I$ is an injective representation of an arbitrary Dynkin quiver. We call such varieties well-behaved quiver Grassmannians \cite{CFR}. We prove the following result. 
\begin{theorem}\label{Thm:WellBehaved}\cite{CFR}
Let $Q$ be a Dynkin quiver. Let $P$ and $I$ be a projective and an injective $Q$-representation, respectively. Let $Z=\Gr_{\mathbf{dim}\,P}(P\oplus I)$. Then $Z$ is a reduced, irreducible and rational locally complete intersection scheme of minimal dimension $\langle\mathbf{dim}\,P,\mathbf{dim}\,I\rangle$. Moreover $Z$ has normal singularities and it is acted upon by an algebraic group $G\subset \textrm{Aut}(P\oplus I)$ with finitely many orbits.
\end{theorem}
We can extend a little bit the class of well-behaved quiver Grassmannians by considering varieties of the form $Z=\Gr_{\mathbf{dim}\,X}(X\oplus Y)$ where $X$ and $Y$ are rigid and such that $[X,Y]^1=0$. In this case the variety $Z$ has the same properties stated in theorem~\ref{Thm:WellBehaved} but not the group action and the normality. Normality does not hold for such quiver Grassmannians, in general, since they can have singularities in codimension~1 (see example~4 in section~\ref{Ex4}).   

\section{Cell decomposition and property~(S)}\label{Sec:CellDec}
A finite partition $(X_i)$ of a complex algebraic variety $X$ is said to be an \emph{$\alpha$--partition} if 
\begin{equation}\label{Eq:DefAlphaPart}
X_1\amalg\cdots\amalg X_i\textrm{ is closed in }X\textrm{ for every }i.
\end{equation}
Clearly, every piece of an $\alpha$--partition is locally closed.  Property \eqref{Eq:DefAlphaPart} can be reformulated by 
\begin{equation}\label{Eq:DefAlphaPart2}
\overline{X}_i\subseteq X_1\amalg\cdots\amalg X_{i-1}
\end{equation}
for every $i$. 
A \emph{cellular decomposition} of $X$ is an $\alpha$--partition whose parts $X_i$ are (complex) affine spaces. By \eqref{Eq:CellDedGrass} and \eqref{Eq:CellDecGrassClosure} we see that the Grassmannian admits an $\alpha$-partition into affine spaces; more precisely the standard cells of the Grassmannian form an $\alpha$-partition with the stronger property that in \eqref{Eq:DefAlphaPart2} the equality holds. The following is an example of a variety admitting a partition into affine spaces which do not admit a cellular decomposition.
\begin{example}\label{Ex:NotAlphaPart}
Let $X=\{[x:y:z]\in\PP^2|\, xyz=0\}$ be the union of three $\PP^1$'s crossing in the three distinct points $[1:0:0]$, $[0:1:0]$ and $[0:0:1]$. We can represent $X$ as a triangle: 
$$
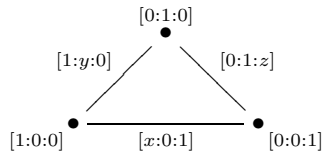
\xymatrix{
&\bullet\ar@{-}^{[0:1:z]}[dr]\ar@{}|(-.18){[0:1:0]}[d]&\\
\bullet\ar@{-}^{[1:y:0]}[ur]&&\bullet\ar@{-}^{[x:0:1]}[ll] \ar@{}^(1.2){[1:0:0]}[ll]\ar@{}^(-0.2){[0:0:1]}[ll]
}
$$
Then $X=\mathcal{C}_1\cup \mathcal{C}_2\cup \mathcal{C}_3$ is the disjoint  union of three affine lines $\mathbb{A}^1$ given by $\mathcal{C}_1=\{[1:y:0]\}$, $\mathcal{C}_2=\{[1:0:z]\}$, $\mathcal{C}_3=\{[0:1:z]\}$ which do not form an $\alpha$-partition. 
\end{example}

The existence of a  cellular decomposition for $X$ is rare but when happens it implies wonderful homological properties:  we denote by $H_i(X)$ the i--th space of the Borel--Moore homology of $X$ (see \cite{CG}). Following \cite[Sec.~1.7]{DLP} we say that  an algebraic variety $X$ has \emph{property~(\textbf{S})}  if:
\begin{itemize}
\item[(\textbf{S}1)] $H_i(X)$ is zero if i is odd and it has no torsion if i is even; 
\item[(\textbf{S}2)] the cycle map $\varphi_i: A_i(X)\rightarrow H_{2i}(X)$ is an isomorphism for all i.  
\end{itemize}
(Here $A_k(X)$ denotes the Chow group generated by $K$--dimensional irreducible subvarieties modulo rational equivalences (see \cite[Sec.~1.3]{Fu})). It is easy to prove (see e.g. \cite[Section~1.10]{DLP}) that
$$
\xymatrix{
\textrm{Cell decompositon}\ar@{=>}[r]&\textrm{Property (S)}
}
$$
but the opposite is not true. A counter-example for the reverse implication was communicated to me by Antonio Rapagnetta: it is the first example of an irrational surface with trivial $H^1$ \cite{Barlow}. Those surfaces have property (S) but they cannot admit a cellular decomposition, since cellular decomposition implies rationality. The variety $X$ of example~\ref{Ex:NotAlphaPart} has non-trivial $H_1$ and hence it cannot admit a cellular decomposition. 

If $X$ is a smooth complex algebraic variety, we can formulate property (S) equivalently in terms of singular cohomology as: 1) the odd cohomology groups $H^{2i+1}(X)$ vanish and 2) the cycle map $A^i(X)\rightarrow H^{2i}(X)$ is an isomorphism. 

By  \cite[Lemma~1.9]{DLP}, if $f:X\rightarrow Y$ is a locally-trivial affine bundle, and $Y$ has property (S) (resp.  admits a cellular decomposition) then $X$ has property (S) (resp. admits a cellular decomposition) too. If $X$ admits an $\alpha$-partition into pieces having property (S) (resp. admitting a cellular decomposition) then $X$ has property (S) (resp. admits a cellular decomposition). 

In \cite{DLP} it is shown that Springer fibers for classical groups admit a cellular decomposition and the Springer fibers for exceptional groups have property (S). It is conjectured that cell decomposition exists also for the exceptional groups.  

We say that a quiver representation $M$ has property $(C)$ (resp. (S)) if every quiver Grassmannian $\Gr_\mathbf{e}(M)$ associated with $M$ admits a cellular decomposition (resp. has property (S)). In \cite{CEFR}  the following result is proved.
\begin{theorem}\cite{CEFR}\label{Thm:CEFR}
Let $Q$ be a connected quiver and $M$ a $Q$--representation. 
\begin{enumerate}
\item If $Q$ is Dynkin, then $M$ has property (C).
\item If $Q$ is affine and $M$ has indecomposable or rigid regular part, then $M$ has property (C). 
\item If $Q$ is arbitrary and $M$ is rigid, then $M$ has property (S). 
\end{enumerate}
\end{theorem}

It is conjectured that every rigid quiver representation has property (C).  For quivers with two vertices this has been proved in \cite{RW}. For quivers of type $\tilde{D}$, theorem~\ref{Thm:CEFR} was partially proved in \cite{LW1} and \cite{LW2}. 
It is not clear to us if the restrictions assumed in part (2) of Theorem~\ref{Thm:CEFR} on the regular part of $M$ for an affine quiver are necessary. At the moment, our proof  only works for those cases. The proof of part $(3)$ is based on a general theorem of Ellingsrund and  Str\o mme \cite{ES} concerning the decomposition of the diagonal in the Chow group. 

\subsection{Decomposition induced by short exact sequences}
In this section we illustrate the idea of the proof of part (1) and part (2) of theorem~\ref{Thm:CEFR}.  Let $Q$ be an acyclic quiver and let 
$$
\eta:\xymatrix{0\ar[r]&M'\ar^\iota[r]&M\ar^\pi[r]&M''\ar[r]&0}
$$
be a short exact sequence in Rep($Q$). This induces the map 
$$
\Psi^\eta: \Gr_\mathbf{e}(M)\rightarrow \coprod_{\mathbf{f+g=e}}\Gr_\mathbf{f}(M')\times\Gr_\mathbf{g}(M''):\; N\mapsto (\iota^{-1}N,\pi(N))
$$ 
between quiver Grassmannians. By taking the preimage $\mathcal{S}_{\mathbf{f,g}}^\eta=(\Psi^{\eta})^{-1}(\Gr_\mathbf{f}(M')\times\Gr_\mathbf{g}(M''))$ of each piece, we get the algebraic map 
$$
\Psi^\eta_{\mathbf{f,g}}: \mathcal{S}_{\mathbf{f,g}}^\eta\rightarrow \Gr_\mathbf{f}(M')\times\Gr_\mathbf{g}(M''):\; N\mapsto (\iota^{-1}N,\pi(N)).
$$ 
The finite partition $
\Gr_\mathbf{e}(M)=\coprod_{\mathbf{f+g=e}}\mathcal{S}_{\mathbf{f,g}}^\eta
$ 
is an $\alpha$-partition (see \cite[Lemma~20]{CEFR}). It is hence natural to investigate the map $\Psi^\eta_{\mathbf{f,g}}$ in order to deduce nice properties of each piece $\mathcal{S}_{\mathbf{f,g}}^\eta$. The first thing to study is its image. It is basically by definition that the image of $\Psi^\eta_{\mathbf{f,g}}$ consists of those pairs $(N',N'')\in \Gr_\mathbf{f}(M')\times\Gr_\mathbf{g}(M'')$ such that in the commutative diagram 
\begin{equation}\label{Eq:CommDiagPsi}
\xymatrix@R=15pt{
\Ext^1(M'',M')\ar@{->>}[r]\ar@{->>}[d]&\Ext^1(M'',M'/N')\ar@{->>}[d]\\
\Ext^1(N'',M')\ar@{->>}[r]&\Ext^1(N'',M'/N')
}
\end{equation}
the element $\eta\in\Ext^1(M'',M')$ is mapped to zero in $\Ext^1(N'',M'/N')$ (see \cite[Lemma~21]{CEFR}). For a general $\eta$ the image of $\Psi^\eta_{\mathbf{f,g}}$ is hence difficult to control. Nevertheless, there are some short exact sequences for which this image is under control. They are called generating.  
\subsection{Generating short exact sequences}
\begin{definition}
An element  $\xi\in\Ext^1(S,X)$ is  \emph{generating} if $\Ext^1(S,X)=\CC\xi$. 
\end{definition}
In other words  $\xi\in\Ext^1(S,X)$ is generating if either $[S,X]^1=0$ and $\xi=0$ or $[S,X]^1=1$ and $\xi\neq 0$.  

If a generating sequence $\xi\in\Ext^1(S,X)$ is split then $\Ext^1(S,X)=0$ and hence the map $\Psi^\xi_{\mathbf{f,g}}$ is surjective by the above description of its image. 

If a generating sequence $\xi\in\Ext^1(S,X)$ is not split then $\Ext^1(S,X)\simeq \CC$ and hence for every $(N_1,N_2)\in\Gr_\mathbf{f}(X)\times \Gr_\mathbf{g}(S)$  the diagram \eqref{Eq:CommDiagPsi} becomes 
\begin{equation}\label{Eq:CommDiagPsiGenerating}
\xymatrix{
\CC\ar@{->>}[r]\ar@{->>}[d]&\Ext^1(S,X/N_1)\ar@{->>}[d]\\
\Ext^1(N_2,X)\ar@{->>}[r]&\Ext^1(N_2,X/N_1)
}
\end{equation}
forcing $[N_2,X]^1\leq 1$, $[S,X/N_1]^1\leq 1$ and $[N_2,X/N_1]^1\leq 1$. A pair $(N_1,N_2)$ is \emph{not} in the image of $\Psi^\xi_{\mathbf{f,g}}$ if and only if $[N_2,X/N_1]^1=1$. By diagram chasing one shows that a pair $(N_1,N_2)$ is \emph{not} in the image of $\Psi^\xi_{\mathbf{f,g}}$ if and only if $[N_2,X]^1=[S,X/N_1]^1=1$. It turns out \cite[Lemma~27]{CEFR} that the following subrepresentations are well--defined
$$
\begin{array}{cc}
X_S:=\textrm{max}\{N\subset X|\, [S,X/N]^1=1\},& S^X:=\textrm{min}\{N\subset S|\,[N,X]^1=1\}.
\end{array}
$$
Let us give a better description of those subrepresentations. The subrepresentation $X_S$ is the maximal subrepresentation of $X$ such that the pushout sequence 
$$
\xymatrix{
\xi:&0\ar[r]&X\ar[r]\ar@{->>}^p[d]&Y\ar[r]\ar@{->>}[d]&S\ar[r]\ar^=[d]&0\\
p_\ast\xi:&0\ar[r]&X/X_S\ar[r]&\overline{Y}\ar[r]&S\ar[r]&0
}
$$
does not split. Dually, the subrepresentation $S^X\subseteq S$ is the minimal subrepresentation such that the pull-back sequence 
$$
\xymatrix{
i^\ast\xi:&0\ar[r]&X\ar[r]\ar^=[d]&\tilde{Y}\ar[r]\ar@{^(->}[d]&S^X\ar[r]\ar@{^(->}^i[d]&0\\
\xi:&0\ar[r]&X\ar[r]&Y\ar[r]&S\ar[r]&0
}
$$
does not split. If $\xi$ is almost split then this description implies that $S^X=S$ and $X_S=0$. We hence say that a generating extension is a \emph{generalized almost split} sequence if $S^X=S$ and $X_S=0$. It is easy to find examples of generalized almost split sequences which are not almost split (see exercise~\ref{Eser:TypeDGenerating}). If $\xi\in \Ext^1(S,X)$ is generating and not split, then, by the Auslander-Reiten formula, $[X,\tau S]=[\tau^- X,S]=1$; Let $f:X\rightarrow \tau S$ and $g:\tau^-X\rightarrow S$ be two non-zero maps then
\begin{equation}\label{Eq:XsSx}
\begin{array}{ccc}
X_S=\textrm{Ker}(f)&\textrm{and}
&
S^X=\textrm{Im}\,(g).
\end{array}
\end{equation}
\begin{example}
Let $Q$ be the equioriented quiver \eqref{Eq:TypeAQuiver} of type $A_n$. Given indices $1\leq i< k<j< \ell\leq n$ let us consider the indecomposable $Q$--representations $X=U_{k,\ell}$ and $S=U_{i,j}$. In view of \eqref{Eq:ExtSpaceTypeA}, $[S,X]^1=1$ and there is a  generating extension 
$
\xi: 0\rightarrow X\rightarrow Y\rightarrow S\rightarrow 0
$ 
where $Y=U_{i,\ell}\oplus U_{k,j}$. A non-zero map between $X$ and $\tau S$ is injective, and hence, by \eqref{Eq:XsSx},  $X_S=0$. The image of a non-zero map $\tau^-X=U_{k-1,\ell-1}\rightarrow U_{i,j}$ is $S^X=U_{k-1,\ell-1}$. For quivers of this type the only generalized almost split sequences are the almost split sequences (see exercise~\ref{Eser:TypeAGenerating}). In general, one can easily find examples of generilized almost split sequences which are not almost split (see exercise~\ref{Eser:TypeDGenerating}). 
\end{example}
Turning back to $\textrm{Im}(\Psi^\xi_{\mathbf{f,g}})$, it follows from the  discussion above that  $(N_1,N_2)$ is \emph{not} in the image of $\Psi^\xi_{\mathbf{f,g}}$ if and only if $N_1\subseteq X_S$ and $N_2\supseteq S^X$. We hence see that for a non-split generating extension $\xi$  the image of $\Psi^\xi_\mathbf{f,g}$ is given by:
\begin{equation}\label{Eq:ImageNonSplitGenerating}
\textrm{Im}\,\Psi^\xi_\mathbf{f,g}=\left(\Gr_\mathbf{f}(X)\times \Gr_\mathbf{g}(S)\right)\setminus \left(\Gr_\mathbf{f}(X_S)\times \Gr_{\mathbf{g-dim}\,S^X}(S/S^X)\right).
\end{equation}

The following technical result, that we call the ``reduction theorem'', provides a way to reduce the problem of checking property (C) or (S) for a $Q$-representation $Y$ which is the center of a generating extension. 

\begin{theorem}\label{Thm:RedThm}
Let $\xi\in\Ext^1(S,X)$ be a generating extension. Then $\Psi_\mathbf{f,g}^\xi$ is a locally trivial affine bundle over its image of rank $\langle\mathbf{g},\mathbf{dim}\,X-\mathbf{f}\rangle$. 
\end{theorem}
Theorem~\ref{Thm:RedThm} is called ``reduction theorem'' because of the following corollary. 
\begin{corollary}\label{Cor:RedThm}
Let $\xi:\xymatrix@1@C=15pt{0\ar[r]&X\ar[r]&Y\ar[r]&S\ar[r]&0}$ be a generating extension. 
If $\textrm{Im}\,\Psi^\xi_\mathbf{f,g}$ admits a cellular decomposition for all $\mathbf{f}$ and $\mathbf{g}$, then  $Y$ has property (C). In particular, if $[S,X]^1=0$ and both $X$ and $S$ have property (C) (or (S)) then $Y$ has property (C) (or (S)).
\end{corollary}
Corollary~\ref{Cor:RedThm} has the following immediate consequence: if $M$ is a preprojective $Q$--representation and all its indecomposable direct summands have property (S) or (C) then $M$ itself has property (S) or (C). This is because the preprojective component is directed. 

In order to prove that every representation of a Dynkin quiver has property (C) it is hence enough to prove it for the indecomposables. This is done by induction using an elementary technique (see \cite[Theorem~45]{CEFR}).

If $Q$ is an  affine quiver then one can use its well-known representation theory to deduce that every indecomposable preprojective $Y$ fits as the middle term of a generating extension $\xi$. Moreover the subrepresentations $X_S$ and $S^X$ are under control and hence is the image of $\Psi_{\mathbf{f,g}}^\xi$. By corollary~\ref{Cor:RedThm} one gets the proof of part (2) of theorem~\ref{Thm:CEFR} by induction.   Proving property (C) for preprojective representations of an arbitrary quiver $Q$ seems to be much harder, due to the fact that the reduction theorem only allow us to use generating extensions. This is done in \cite{RW} for quivers with two vertices by combining the reduction theorem~\ref{Thm:RedThm} with covering theory.

The proof of theorem~\ref{Thm:CEFR} has the following easy corollaries.
\begin{corollary}
Given an acyclic quiver $Q$, every $Q$-representation $M$ whose regular part is rigid has property (S). 
\end{corollary}
\begin{proof}
Decompose $M=M_\mathcal{P}\oplus M_\mathcal{R}\oplus M_\mathcal{I}$ as the sum of a preprojective, a regular and a preinjective $Q$--representation. By  theorem~\ref{Thm:CEFR}, every indecomposable direct summand of both $M_\mathcal{P}$,  $M_\mathcal{R}$ and $M_\mathcal{I}$ has property (S). By the reduction theorem~\ref{Thm:RedThm} it follows that $M_\mathcal{P}$,  $M_\mathcal{R}$ and $M_\mathcal{I}$ have property (S). Since $[M_\mathcal{P},M_\mathcal{R}\oplus M_\mathcal{I}]^1=[M_\mathcal{R}, M_\mathcal{I}]^1=0$ again by the reduction theorem~\ref{Thm:RedThm} we get that $M$ has property (S). 
\end{proof}
\begin{corollary}\cite[Corollary~42]{CEFR}
Let $M$ be a rigid representation of a quiver $Q$ and let  $\iota: \Gr_\mathbf{e}(M)\rightarrow \prod_{i\in Q_0}\Gr_{\mathbf{e}_i}(M_i)$ be the closed embedding. Then the induced map in cohomology $\iota^\ast:H^\bullet(\prod_{i\in Q_0}\Gr_{\mathbf{e}_i}(M_i))\rightarrow H^\bullet(\Gr_\mathbf{e}(M))$ is surjective. 
\end{corollary} 

\begin{corollary}\cite[Corollary~2]{CEFR}
Let $M$ be a rigid representation of a quiver $Q$ and let $X=\Gr_\mathbf{e}(M)$ be a quiver Grassmannian attached to it. Then $X$ is defined over $\ZZ$ and it has polynomial point count, i.e.
$$
\# \Gr_\mathbf{e}(M)|_{\mathbf{F}_q}=\sum_i\mathrm{dim}_\mathbb{Q} H^{2i}(\Gr_\mathbf{e}(M),\mathbb{Q})q^i.
$$ 
\end{corollary}

\section{The cluster multiplication formula}\label{Sec:CCMap}
In this section we provide an application of the reduction theorem~\ref{Thm:RedThm} to cluster algebras. Let $Q$ be an acyclic quiver with $n$ vertices and let $M$ be a $Q$-representation. The \emph{F-polynomial} of $M$ is 
$$
F_M(\mathbf{y})=\sum_\mathbf{e}\chi(\Gr_\mathbf{e}(M))\mathbf{y}^\mathbf{e}\in\ZZ[y_1,\cdots, y_n]
$$
where $\chi$ denotes the Euler-Poincar\`e characteristic. The \emph{$\mathbf{g}$--vector} or \emph{index} of $M$ is 
$$
\mathbf{g}_M=[I_1^M]-[I^M_0]\in K_0(\textrm{Rep}(Q))\simeq \ZZ^{Q_0}
$$
where $0\rightarrow M\rightarrow I_0^M\rightarrow I_1^M\rightarrow 0$ is the minimial injective resolution of $M$. Notice that $(\mathbf{g}_M)_i=-\langle S_i,M\rangle$. The \emph{exhange matrix} $B=(b_{i,j})_{i,j\in Q_0}\in \textrm{Mat}_{n\times n}(\ZZ)$ of $Q$ is the integer matrix given by
$$
b_{i,j}=\#\{j\rightarrow i\in Q_1\}-\#\{i\rightarrow j\in Q_1\}.
$$
The cluster character of $M$ is the Laurent polynomial 
$$
CC_M(\mathbf{x},\mathbf{y})=\sum_\mathbf{e}\chi(\Gr_\mathbf{e}(M))\mathbf{x}^{B\mathbf{e}+\mathbf{g}_M}\mathbf{y}^\mathbf{e}=:CC(M)\in \ZZ[y_1,\cdots, y_n][x_1^{\pm1},\cdots, x_n^{\pm1}].
$$
The reduction theorem~\ref{Thm:RedThm} implies the following multiplication formula. To state the precise result we need to recall that given a generating extension $\xi\in\Ext^1(S,X)$ there exists an exact sequence $0\rightarrow X/X_S\rightarrow \tau S^X\rightarrow I\rightarrow 0$ where $I$ is injective. 
\begin{corollary}\cite[Theorem~66]{CEFR}
Let $\xi:0\rightarrow X\rightarrow Y\rightarrow S\rightarrow 0$ be a generating extension. Then 
\begin{equation}\label{Eq:MultForm}
CC(X)CC(S)=CC(Y)+\mathbf{y}^{\mathbf{dim}\,S^X}CC(X_S)CC(S/S^X)\,\mathbf{x}^\mathbf{f}
\end{equation}
where $I=\oplus_{j\in Q_0}I_j^{f_j}$. 
\end{corollary}
The multiplication formula \eqref{Eq:MultForm} is a slight generalization of the multiplication formula of Caldero and Keller \cite{CK2} and it can be interpreted as a categorification of the exchange relations in the cluster algebra associated with $Q$. The apperence of $\mathbf{dim}\,S^X$ in the formula seems to be new and provides a representation theoretic description of the $\mathbf{c}$--vectors. 
\newpage
\section{Exercises}\label{Sec:Exercises}
We conclude the notes with a list of exercises divided by arguments.
\subsection{Generalities on quiver Grassmannians}
\begin{eser}
Recall that the complete flag variety $\mathcal{F}l_{n+1}=\{U_1\subset U_2\subset\cdots\subset U_n\subset \CC^{n+1}|\, \textrm{dim }U_i=i\}$ can be realized as the quiver Grassmannian $\Gr_\mathbf{e}(P_1^{n+1})$, where $P_1$ is the projective cover of the simple $S_1$ for the equioriented type A quiver $Q: 1\rightarrow\cdots\rightarrow n$ and $\mathbf{e}=(1,2,\cdots, n)$. Show that the dimension of $\mathcal{F}l_n$ is $\langle\mathbf{e},\mathbf{d}-\mathbf{e}\rangle$ where $\mathbf{d}=\mathbf{dim}\,P_1^{n+1}$. 
\end{eser}

\begin{eser}
Consider the quiver $Q:1\rightarrow 2$ and let $M=P_1\oplus S_1\oplus S_2$. Show that $\Gr_{(1,1)}(M)$ is isomorphic to two $\PP^1$ crossing in one point.
\end{eser}
\begin{eser}
Given an e-subset $I=(1\leq i_1<i_2<\cdots<i_e\leq n)$ of $[1,n]$, compute the dimension of the affine space $\mathcal{C}_I\subset \Gr_e(\mathbb{C}^d)$.
\end{eser}

\begin{eser}
Show that the universal quiver Grassmannian $\Gr_\mathbf{e}^Q(\mathbf{d})$ has dimension $\langle\mathbf{e},\mathbf{d-e}\rangle+\textrm{dim}\,R_\mathbf{d}(Q)$.
\end{eser}

\begin{eser}
Given $g=(g_i)\in G_\mathbf{d}$, $N=(N_i)\in\Gr_\mathbf{e}(\mathbf{d})$ and $M\in R_\mathbf{d}(Q)$ consider the action: $g\cdot N=(g_i(N_i))_{i\in Q_0}$ and $g\cdot M=(g_jM_\alpha g_i^{-1})_{\alpha:i\rightarrow j\in Q_1}$. Show that the universal quiver Grassmannian $\Gr_\mathbf{e}^Q(\mathbf{d})\subset \Gr_\mathbf{e}(\mathbf{d})\times R_\mathbf{d}(Q)$ is invariant under the diagonal $G_\mathbf{d}$ action and that the map $p_\mathbf{d}:\Gr_\mathbf{e}^Q(\mathbf{d})\rightarrow R_\mathbf{d}(Q)$ is $G_\mathbf{d}$-equivariant. 
\end{eser}

\begin{eser}
The group $G_\mathbf{e}$ acts on $\Hom(\mathbf{e},\mathbf{d})$ by: $g=(g_i)\in G_\mathbf{e}$, $(N,f)\in \Hom(\mathbf{e},\mathbf{d})$, $\mathbf{g}\cdot (N,f):=((g_{t(\alpha)}N_\alpha g_{s(\alpha)}^{-1}),(f_ig_i^{-1}))$. Verify that given $M\in R_\mathbf{d}(Q)$, $\Hom(\mathbf{e},M)\subset \Hom(\mathbf{e},\mathbf{d})$ is $G_\mathbf{e}$-stable. Prove that $G_\mathbf{e}$ acts freely on $\Hom^0(\mathbf{e},M)$.
\end{eser}

\begin{eser}
Let $M$ be a rigid quiver representation. Using the tangent space formula, show that every non-empty quiver Grassmannian $\Gr_\mathbf{e}(M)$ attached to $M$ is smooth of minimal dimension $\langle\mathbf{e},\mathbf{dim}\,M-\mathbf{e}\rangle$.
\end{eser}
\begin{eser}
Let $M$ be a rigid quiver representation of dimension vector $\mathbf{d}$. Show that $\Gr_\mathbf{e}(M)$ is non--empty if and only if $\Gr_\mathbf{e}(M')$ is non--empty for every $M'\in R_\mathbf{d}(Q)$. 
\end{eser}
\begin{eser}
Consider the quiver $Q:1\rightarrow 2$ and let $M^1=P_1\oplus P_2\oplus I_1\oplus I_2$. Find the generic subrepresentation type of the degenerate flag variety $\Gr_{(1,2)}(M^1)$. 
\end{eser}
\begin{eser}\label{Es:GEnExt}
Let $X$ and $Y$ be rigid representations of an acyclic quiver $Q$ such that $[X,Y]^1=0$ and let $M$ be a $Q$--representation such that $M\leq_{deg}X\oplus Y$. Use Bongartz's theorem~\ref{Thm:Bongartz} to show that there exists a short exact sequence $0\rightarrow X\rightarrow M\rightarrow Y\rightarrow 0$.
\end{eser}
\begin{eser}
Consider the quiver $Q:1\rightarrow 2$ and let $M^2=P_1\oplus S_1^2\oplus S_2^2$. Find all iso-strata of the mf-liner degeneration of the flag variety $\Gr_{(1,2)}(M)$ and show that there are two generic subpresentation types. [Hint: compute the dimension of the iso-strata]
\end{eser}
\begin{eser}
Let $Q$ be an acyclic quiver and let $M=P\oplus I$ where $P$ is projective and $I$ is injective. Consider the quiver Grassmannian $\Gr_{\mathbf{e}}(M)$ where $\mathbf{e}=\mathbf{dim}\,P$. Let $N\in\Gr_{\mathbf{dim}\,P}(M)$. Show that the iso-stratum $\mathcal{S}_{[N]}$ has dimension less or equal than $\langle\mathbf{e},\mathbf{dim}\,M-\mathbf{e}\rangle$. Using Bongartz's theorem, show that the only generic iso-stratum is $\mathcal{S}_{[P]}$. Conclude that $\Gr_\mathbf{e}(M)$ is irreducible of minimal dimension. 
\end{eser}

\subsection{Quiver Grassmannians of type $A$}
\begin{eser}
Find a normal form for pairs of matrices $(A,B)\in\mathrm{Mat}_{m\times k}\times\rm{Mat}_{k\times n}$ by base change. In other words, find the decomposition of a representation $V$ of $Q:1\rightarrow 2\rightarrow 3$ as direct sum of indecomposable $Q$--representations. 
\end{eser}
\begin{eser}
Let $Q=1\rightarrow 2$ and let $M=S_2\oplus P_1\oplus S_1\in \textrm{Rep}(Q)$. Order the indecomposable direct summands of $M$ as $M(1)=S_1$, $M(2)=P_1$ and $M(3)=S_2$. Consider the quiver Grassmannian $\Gr_{(1,1)}(M)$ and the $\mathbb{C}^\ast$-action given by $\lambda\cdot m=\lambda^{k-1}m$ for every $m\in M(k)$. Prove that there is a torus fixed point $L$ such that its attracting set $\mathcal{C}_L$ is not an affine space. 
\end{eser}
\begin{eser}
Let $Q$ be a Dynkin quiver, and let   $X$ and $Y$ be two rigid $Q$--representations such that $[X,Y]^1=0$. Show that the dimension of each isostratum $\mathcal{S}_{[N]}$ of a quiver Grassmannian $\Gr_{\mathbf{dim}\,X}(X\oplus Y)$ satisfies: 
$$
\textrm{dim}\,\mathcal{S}_{[N]}\leq \langle\mathbf{dim}\,X,\mathbf{dim}\, Y\rangle.
$$
and equality holds if and only if $N\simeq X$. ([HINT: use the fact that the degeneration order for Dynkin quivers is equivalent to the Hom-order: $M\leq_{deg} M'$ if and only if $[M,L]\leq [M',L]$, for every $L$]). 
Conclude that $\Gr_{\mathbf{dim}\,X}(X\oplus Y)$ is irreducible of minimal dimension. 
\end{eser}

\begin{eser}
Realize the degenerate flag variety $\mathcal{F}l_4^a$ as a Schubert variety. 
\end{eser}
\begin{eser}
Let $Q:1\rightarrow 2\rightarrow 3$ and $M=P_3\oplus P_2^2\oplus S_2\oplus I_2\oplus I_1$. Verify that $M$ is catenoid and describe the natural embedding of $\Gr_{(1,2,1)}(M)$ inside a partial flag manifold. Show that there are two irreducible components.  
\end{eser}
\begin{eser}
Let $Q$ be the equioriented quiver of type $A_n$, and let $A=KQ$ be its path algebra. Put $\mathbf{e}=\mathbf{dim}\,A$ and  $\mathbf{d}=\mathbf{dim}\,(A\oplus DA)$ as in section~\ref{Sec:LinDeg}. Use theorem~\ref{Thm:CFFFR} and exercise~\ref{Es:GEnExt} to show that a point $M\in R_\mathbf{d}(Q)$ belongs to $\mathcal{U}_{\textrm{flat,Irr}}$ if and only if there exists a short exact sequence $0\rightarrow A\rightarrow M\rightarrow DA\rightarrow 0$. 
\end{eser}
\subsection{Cellular decomposition of quiver Grassmannians}
\begin{eser}
Let $\eta:\xymatrix{0\ar[r]&\tau S\ar^\iota[r]&Y\ar^\pi[r]&S\ar[r]&0}$ be an almost split sequence. Describe the image of the map 
$$
\xymatrix@R=3pt{
\Psi_{\mathbf{f,g}}^\eta:&\Gr_\mathbf{e}(E)\ar[r]& \Gr_\mathbf{f}(\tau S)\times\Gr_\mathbf{g}(S)\\
&N\ar@{|->}[r]&(N\cap \iota(\tau S), \pi(N))
}
$$
\end{eser}
\begin{eser}
Show that if $\xi\in\Ext^1(S,X)$ is generating then 
$$
\textrm{Im}(\Psi_{\mathbf{f,g}}^\xi)=\{(N_1,N_2)\in\Gr_\mathbf{f}(X)\times\Gr_\mathbf{g}(S)|\, [N_2,X/N_1]^1=0\}.
$$ 
\end{eser}

\begin{eser}
Let $\xi\in\Ext^1(S,X)$ be a non--split generating extension. Prove that $X_S$ and $S^X$ are well-defined as follows. (For $X_S$) for every $N,N'\subset X$ such that $[S, X/N]^1=[S,X/N']^1=1$, one has $[S, X/(N+N')]^1=1$. Dually (for $S^X$)
for every $N,N'\subset S$ such that $[N, X]^1=[N',X]^1=1$, one has  $[N\cap N', X]^1=1$. 
\end{eser}

\begin{eser}\label{Eser:TypeDGenerating}
Prove that an almost split sequence $\xi\in\Ext^1(X,\tau X)$ ending in a  \emph{brick} $X$ (i.e. a $Q$--representation such that $[X,X]=1$) is generalized almost split. Find an example of a generalized almost split which is not almost split. (Hint: look among the representations of a quiver of type $D_4$.) 
\end{eser}

\begin{eser}\label{Eser:TypeAGenerating}
Let $Q$ be the equioriented quiver of type $A_n$. Prove that a non-split generating extension (between two indecomposable  $Q$-representations) is generalized almost split if and only if it is almost split. 
\end{eser}
\begin{eser}
Prove the equalities in \eqref{Eq:XsSx}. 
\end{eser}
\begin{eser}
Let $X$ and $Y$ be indecomposable preprojectives, such that $\Hom(X,Y)=\CC\iota$ with $\iota:X\rightarrow Y$ an irreducible monomorphism. Let $S=Coker(\iota)$. Show that the short exact sequence 
$$
\xi:\,\xymatrix{0\ar[r]&X\ar^\iota[r]&Y\ar^\pi[r]&S\ar[r]&0}
$$
induced by $\iota$ is generating and $S^X=S$. 
\end{eser}

\begin{eser}
A short exact sequence 
$$\xymatrix@C=30pt{0\ar[r]&A\ar^(.4){f=(f_1,f_2)^t}[r]&B_1\oplus B_2\ar^(.6){g=(g_2,g_1)}[r]&C\ar[r]&0}$$ gives rise to a commutative diagram 
$$
\xymatrix{
&B_1\ar^{g_2}[dr]&\\
A\ar^{f_1}[ur]\ar_{-f_2}[dr]&&C\\
&B_2\ar_{g_1}[ur]&
}
$$
which is both a push--out and a pull-back square. Prove that
$$
\begin{array}{ccc}
Ker(g_i)\simeq Ker(f_i),& Coker(g_i)\simeq Coker(f_i)&(i=1,2).
\end{array}
$$
\end{eser}
\begin{eser}
Let $Q$ be the following quiver of type $\tilde{A}_2$:$$\xymatrix{&2\ar^\beta[dr]&\\1\ar^\alpha[ur]\ar_\gamma[rr]&&3}$$
Let $M$ be the indecomposable $Q$ representation of dimension vector $(3,3,4)$. Using the construction seen during the lecture, find a cellular decomposition of the quiver Grassmannian $\Gr_\mathbf{e}(M)$ for $\mathbf{e}=(1,2,3)$. 
\end{eser}

\begin{eser}
Let $Q$ be the following quiver of type $\tilde{D}_4$: 
$$
\xymatrix@R=3pt{
&1\\
&2\\
0\ar[uur]\ar[ur]\ar[dr]\ar[ddr]&\\
&3\\
&4
}
$$
and let $M$ be the indecomposable preprojective $Q$--representation of dimension vector $(3,2,2,2,2)$. Let $\mathbf{e}=(1,1,1,1,1)$. Find a cellular decomposition of the quiver Grassmannian $\Gr_\mathbf{e}(M)$, using the techniques seen at the lecture. Find a geometric interpretation. 
\end{eser}
\begin{eser}
Let $X$ be a rigid brick (i.e. $[X,X]=1$ and $[X,X]^1=0$) which is not projective and let $\xi: 0\rightarrow \tau X\rightarrow E\rightarrow X\rightarrow 0$ be the almost split sequence ending in $X$. Show that $E$ is rigid and $[X\oplus \tau X,E]^1=[E,X\oplus \tau X]^1=0$. Show that $\xi$ is generalized almost split. 
\end{eser}

\begin{eser}
Let $\xi:0\rightarrow X\rightarrow Y\rightarrow S\rightarrow 0$ be a generating extension. Prove that the reduction theorem~\ref{Thm:RedThm} implies the following multiplication formula of $F$-polynomials
$
F_XF_S=F_Y+\mathbf{y}^{\mathbf{dim}S^X}F_{X_S}F_{S/S^X}.
$
\end{eser}


\bibliographystyle{amsplain}

\end{document}